\documentclass[smallextended,nospthms,runningheads]{svjour3}
\usepackage[T1]{fontenc}
\usepackage[utf8]{inputenc}
\usepackage{newtxtext}
\usepackage{amsmath,amssymb,amsthm,mathtools}
\usepackage{newtxmath}
\usepackage{booktabs,array}
\usepackage{microtype}
\usepackage[numbers,sort&compress]{natbib}
\usepackage{xcolor}
\usepackage{ptrf-style}
\usepackage[colorlinks=true,linkcolor=blue!50!black,citecolor=blue!50!black,urlcolor=blue!50!black]{hyperref}
\usepackage{bookmark}
\newcommand{\doi}[1]{\url{https://doi.org/#1}}
\numberwithin{equation}{section}
\newtheoremstyle{ptrfplain}{7pt plus 2pt minus 2pt}{7pt plus 2pt minus 2pt}
  {\itshape}{}{\sffamily\bfseries}{}{0.5em}{}
\newtheoremstyle{ptrfdefinition}{7pt plus 2pt minus 2pt}{7pt plus 2pt minus 2pt}
  {\normalfont}{}{\sffamily\bfseries}{}{0.5em}{}
\theoremstyle{ptrfplain}
\newtheorem{theorem}{Theorem}[section]
\newtheorem{lemma}[theorem]{Lemma}
\newtheorem{proposition}[theorem]{Proposition}
\newtheorem{corollary}[theorem]{Corollary}
\theoremstyle{ptrfdefinition}

\newtheorem{remark}[theorem]{Remark}

\newcommand{\N}{\mathbb N}
\newcommand{\R}{\mathbb R}
\newcommand{\E}{\mathbb E}
\newcommand{\Prob}{\mathbb P}
\newcommand{\1}{\mathbf 1}
\newcommand{\FV}{Fleming--Viot}
\newcommand{\diag}{\operatorname{diag}}

\newcommand{\spanof}{\operatorname{span}}
\newcommand{\eps}{\varepsilon}

\allowdisplaybreaks[2]
\journalname{Probability Theory and Related Fields}
\title{Ancestral circulation beyond two \FV\ particles}
\titlerunning{Ancestral circulation beyond two Fleming--Viot particles}
\author{Yaohui Lin \and Tiankui Lu}
\authorrunning{Y. Lin and T. Lu}
\institute{Yaohui Lin  \at
  School of Mathematical Sciences, South China Normal University,
  Guangzhou, China\\
  \email{linyuan@m.scnu.edu.cn} (corresponding author)
  \and
  Tiankui Lu \at
  Information and Artificial Intelligence, Anhui Agricultural University,
  Hefei, China\\
  \email{lutiankui@stu.ahau.edu.cn}}
\date{}
\hypersetup{pdftitle={Ancestral circulation beyond two Fleming-Viot particles},pdfauthor={Yaohui Lin and Tiankui Lu}}
\begin{document}
\maketitle
\begin{abstract}
Reversible motion can produce a directed circulation of surviving ancestry
in a \FV\ particle system. We prove that two particles are the exact
threshold for universal ancestral balance in finite systems with uniform
copying. For two particles, the stationary current of surviving labels
vanishes for every finite irreducible reversible Markov chain and every positive
killing vector. On a fixed complete three-state graph, nonzero circulation
occurs for each population size at least three on an open dense set of
killing vectors of full Lebesgue measure. A single killing vector produces
circulation simultaneously for every population size at least three, and
such vectors form a dense $G_\delta$ set of full measure. The current is
divergence free and vanishes on bridges, so three sites and three particles
are the smallest possible setting for circulation. We also identify the
probabilistic effect of this circulation: for sufficiently rare transplants,
migration followed by fixation is more likely in one direction than in
the reverse direction.
\keywords{\FV\ particle system \and Ancestral current \and Fixation probability
\and Reversibility \and Nonexplosion \and Green kernel}
\subclass{60J27 \and 60J80 \and 60J10}
\end{abstract}

\section{Introduction}\label{sec:introduction}

Consider $N$ particles which move independently according to a reversible
continuous time Markov chain on a finite set $I$. A particle at $i\in I$
is killed at rate $k_i>0$; it immediately moves to the position of a
uniformly chosen survivor and takes that survivor's ancestral label.
This is the \FV\ system. Starting with distinct labels, eventually all
particles have the same label. We ask whether the surviving ancestry
inherits the balance of the underlying motion, or whether selection can
force it around a directed cycle.

A reversible Markov chain has balanced edge fluxes at stationarity.
Ancestral selection can turn this balance into circulation. For two
particles, the first death determines the eventual ancestor. With three
or more particles, ancestry is selected through successive reproduction
events. This change is decisive: spatially varying killing can turn
reversible motion into a directed ancestral circulation. Our main result
determines its exact population threshold and establishes the phenomenon
for every larger population on a three-state motion graph. The
two-particle case rests on a duality known from the nonexplosion theory;
larger populations require a different mechanism.

Let $Q$ be the motion generator, with off-diagonal rates $q(i,j)$, and
suppose that $Q$ is irreducible and reversible with respect to positive
weights $\nu=(\nu_i)_{i\in I}$:
\begin{equation}\label{eq:reversibility}
 \nu_iq(i,j)=\nu_jq(j,i),\qquad
 (Qf)(i)=\sum_{j\ne i}q(i,j)\bigl(f(j)-f(i)\bigr).
\end{equation}
We assume $|I|\ge2$ and $N\ge2$. Each killed particle chooses its donor
among the other $N-1$ particles. Copying particles at the same site can
change their labels and is included in the dynamics.

We weight stationary label-carrying jumps by the probability that the
label survives. To define the resulting transport coefficients, write
$e_i$ for the unit vector at $i$, $|c|=\sum_i c_i$, and
$\N_0=\{0,1,2,\ldots\}$. The population counts form an irreducible chain
on $\Omega_N=\{c\in\N_0^I:|c|=N\}$, with stationary probability $\pi_N$.
For $|d|=N-1$, let $h_N(i,d)$ be the probability that the label of one
specified particle at $i$ ultimately occupies the whole population when
the remaining particles have counts $d$. Define
\begin{equation}\label{eq:transport}
 F_{ij}^{(N)}=\sum_{c:c_i>0}\pi_N(c)c_i h_N(j,c-e_i).
\end{equation}
Thus $q(i,j)F_{ij}^{(N)}$ is the stationary intensity of jumps from $i$
to $j$, weighted by the probability that the moving particle's label
survives. The difference between opposite intensities is the ancestral
current:
\begin{equation}\label{eq:current}
 J_{ij}^{(N)}=q(i,j)F_{ij}^{(N)}-q(j,i)F_{ji}^{(N)}.
\end{equation}
The coefficient $F_{ij}^{(N)}$ is defined for every pair of sites. In
expressions involving $c_i h_N(j,c-e_i)$, the term is set to zero when
$c_i=0$, before evaluating the argument of $h_N$.

The main result gives a sharp distinction between two particles and
larger populations. For the second part, the motion is fixed on the
complete three-state graph:
\begin{equation}\label{eq:triangle}
 I=\{0,1,2\},\qquad q(i,j)=1\ (i\ne j),\qquad \nu_i=1.
\end{equation}
Put $\Delta_N(k)=F_{02}^{(N)}-F_{20}^{(N)}$ on this graph.

\begin{theorem}\label{thm:main}
For every finite reversible motion \eqref{eq:reversibility} and every
positive killing vector, the two-particle transport satisfies
\begin{equation}\label{eq:two-balance-main}
 \frac{F_{ij}^{(2)}}{\nu_i}=\frac{F_{ji}^{(2)}}{\nu_j}
 \qquad(i,j\in I).
\end{equation}
In particular, all two-particle edge currents vanish.

For the three-state motion \eqref{eq:triangle} and every $N\ge3$, the set
\begin{equation}\label{eq:generic-set}
 \mathcal U_N=\{k\in(0,\infty)^3:\Delta_N(k)\ne0\}
\end{equation}
is open and dense and has full Lebesgue measure. Its complement is the
intersection of $(0,\infty)^3$ with a proper algebraic set. Furthermore,
\begin{equation}\label{eq:common-set}
 \mathcal U_* = \bigcap_{N=3}^{\infty}\mathcal U_N
\end{equation}
is a dense $G_\delta$ set of full measure. Every $k\in\mathcal U_*$ gives
nonzero circulation in the three-state system for every $N\ge3$.
\end{theorem}

The population threshold is sharp: at $N=2$ the current vanishes
identically in $k$ on every reversible graph, while at $N=3$ it is already
nonzero at a generic $k$ on the triangle. At constant killing, every label
fixes with probability $1/N$ and all edge currents vanish.

\begin{remark}
For every finite motion graph and every $N\ge2$, the current satisfies
\begin{equation}\label{eq:divergence-main}
 \sum_{j\ne i}J_{ij}^{(N)}=0\qquad(i\in I).
\end{equation}
This follows from the stationarity of the population and the fixation
equation; see Proposition~\ref{prop:conservation}. It forces the current
to vanish on every bridge. On the triangle it gives
\begin{equation}\label{eq:triangle-orientation}
 J_{02}^{(N)}=J_{21}^{(N)}=J_{10}^{(N)}=\Delta_N.
\end{equation}
Consequently, three sites and three particles are the smallest possible
numbers for nonzero edge circulation.
\end{remark}

For three particles the phenomenon is visible in exact arithmetic. On
the triangle, the stationary and fixation equations can be solved
completely, and at $k=(1,2,4)$ they give
\begin{equation}\label{eq:three-witness-intro}
 \Delta_3(1,2,4)=\frac{421413160}{139117744582101}>0.
\end{equation}
The full solution is given in Appendix~\ref{app:three}. For larger
populations the proof proceeds differently, by a perturbation of
constant killing whose first nonzero term has degree three; see
Section~\ref{sec:idea}.

The population threshold is connected to the nonexplosion problem for
\FV\ systems. This asks whether the successive branching times can
accumulate in finite time. Kwa\'{s}nicki~\citep{Kwasnicki2024} proved
nonexplosion for two particles driven by a symmetric Hunt process with
finite reference measure, for almost every initial configuration. He
explicitly raised the question for more than two particles in
\citep[Remark~1.2]{Kwasnicki2024}.
Using his Green-kernel invariant measure, we derive the pairwise balance
identity \eqref{eq:two-balance-main}. This gives a finite-population
counterpart of the symmetry behind the two-particle result. Our theorem
settles the corresponding transport question:
universal ancestral balance ends at two particles, and circulation is
generic beyond that threshold on the fixed three-state graph.
Proposition~\ref{prop:rare-success} turns this circulation into a
directional bias for migration and fixation.

\FV\ particle systems of the present type were introduced in the study of interacting diffusions
\citep{BurdzyHolystIngermanMarch1996,BurdzyHolystMarch2000} and are used
to approximate laws conditioned on survival
\citep{FerrariMaric2007,Villemonais2014}. Their genealogy gives a
surviving lineage, or spine, whose law has been studied at fixed
population size and in the large-population limit
\citep{BieniekBurdzy2018,BurdzyEnglander2024}.
Tagged-particle limits and scaling limits of inherited genetic information
have been established in \citep{GrigorescuKang2006} and
\citep{Tough2023}, respectively.
For two Brownian particles started at the same point on the positive
half-line, the spine and Brownian motion conditioned to stay positive
have mutually singular laws
on infinite-time path space
\citep{BurdzyTadic2024}. The same separation occurs in a bounded interval
\citep{BurdzyEnglanderMarshall2025}.

The approximation of conditioned laws has also been studied
quantitatively. In finite spaces, stationary population laws admit laws
of large numbers and Gaussian fluctuations
\citep{AsselahFerrariGroisman2011,LelievrePillaudVivienReygner2018}.
C\'{e}rou, Delyon, Guyader and Rousset~\citep{CerouDelyonGuyaderRousset2020}
proved a central limit theorem at a fixed time, including hard killing.
Quantitative mixing and uniform in time approximation estimates have
been obtained for discrete systems under a contraction condition
\citep{CloezThai2016Quantitative} and for hard-killing diffusions in a
metastable regime \citep{JournelMonmarche2025}.
Cloez and Corujo~\citep{CloezCorujo2022} proved uniform in time
propagation of chaos for Moran models with additive selection under
exponential stability of the normalized semigroup.
A recent preprint by Cardaliaguet, Cirant, Jackson and
Souganidis~\citep{CardaliaguetCirantJacksonSouganidis2026} establishes
uniform in time weak convergence for Brownian systems with hard killing
in smooth bounded domains.
Here the fixation probabilities allow us to extract the spatial current
of surviving labels directly from the stationary population.

Section~\ref{sec:idea} explains the proof of Theorem~\ref{thm:main}.
We develop the fixation equations in Section~\ref{sec:founder} and
prove conservation and two-particle balance in
Section~\ref{sec:balance}. Section~\ref{sec:cubic} computes the
three-state current and completes the main proof.
The rare-transplant interpretation is given in
Section~\ref{sec:transplants}. Section~\ref{sec:discussion} states the
open multiparticle question precisely and describes its relation to the
finite theory.
Appendices~\ref{app:three}--\ref{app:estimates},
following the references, contain the explicit three-particle solution,
the perturbation calculations, and the finite estimates for transplants.

\section{Idea of the proof}\label{sec:idea}

We first describe the balance mechanism for two particles. Let
$A=\diag(k)-Q$, let $P_t=e^{-tA}$ be the killed motion semigroup, and
write $p_t=P_t\1$ for its survival probabilities. Reversibility makes
the Green kernel $G=A^{-1}$ symmetric with respect to $\nu$. The
stationary law of the ordered pair is
\[
 \mu(x,y)=\frac{\nu_xG(x,y)}{Z},\qquad
 Z=\sum_x\nu_x(G\1)(x).
\]
This is Kwa\'{s}nicki's invariant measure
\citep[Theorem~1.5]{Kwasnicki2024}. The label of one particle survives
exactly when the other particle is killed first. Expressing this
probability with $P_t$ and averaging against $\mu$ gives
\[
 F_{ij}^{(2)}=\frac{2\nu_i}{Z}
                 \int_0^\infty p_t(i)p_t(j)\,dt.
\]
Indeed, if $\ell_t=P_tk$ denotes the lifetime density, then
$G\ell_t=P_tGk=P_t\1=p_t$. The remaining product is symmetric in $i,j$.
This proves the balance identity. Section~\ref{sec:balance} gives the
calculation.

For three particles we can solve the stationary and fixation equations
on the triangle explicitly. At $k=(1,2,4)$ there are ten population
count states and eighteen states for a distinguished particle and the
other two particles. The solution gives the explicit value
\eqref{eq:three-witness-intro}. The equations and all their solution
values are given in Appendix~\ref{app:three}. Conservation then
identifies this positive difference with circulation around
$0\to2\to1\to0$.

For larger populations the size of a direct stationary system grows
with $N$. We instead perturb constant killing, putting
$k=(1-\eps,1,1+\eps)$, and compute the first nonzero term of the current.
Both the fixation probabilities and the stationary law vary with
$\eps$. Their linear contributions cancel. Reflection of sites $0$ and
$2$ makes the current odd in $\eps$, and the cubic term is
\begin{equation}\label{eq:cubic-intro}
\begin{split}
 \Delta_N(1-\eps,1,1+\eps)&=-C_N\eps^3+O(\eps^5),\\
 C_N&=\frac{2(N-3)(N-2)(N-1)^2}
 {9N(3N-2)^2(3N-1)^2(3N+1)},\qquad N\ge4.
\end{split}
\end{equation}
In particular, $C_N>0$ for every $N\ge4$.

The useful feature of this calculation is that its degree stays bounded.
The first three fixation coefficients are polynomials of degree at most
three in the counts of the other particles. For the stationary response,
the neutral generator preserves polynomial degree and acts by a nonzero
scalar on each highest homogeneous part of positive degree. We solve the resulting Poisson
equations by descending through the degrees. A perturbation raises the
degree by at most one, so all terms in the cubic current reduce to
neutral moments of degree at most four. This yields the formula in
\eqref{eq:cubic-intro} for every $N\ge4$, using the same calculation.

Finally, the stationary probabilities and the fixation probabilities are
rational functions of the killing rates. Thus a single nonzero value for
a fixed $N$ makes the zero set of $\Delta_N$ a proper algebraic set.
Equation~\eqref{eq:three-witness-intro} provides that value for $N=3$,
and \eqref{eq:cubic-intro} provides it for every $N\ge4$. Taking the
countable intersection of the full-measure nonzero sets gives a common
killing vector for all these population sizes.

\section{Fixation probabilities and the founder equation}\label{sec:founder}

In this section we derive a linear equation for the founder probability
$h_N$ and show that its solution depends rationally on the rates. We begin
with the labelled construction, which also gives the fixation estimates
used below.

Number the particle slots from $1$ to $N$. A particle in slot $p$ at $i$
moves to $j$ at rate $q(i,j)$. For each ordered pair of distinct slots
$(p,r)$, slot $p$ copies the position and label of slot $r$ at rate
$k_i/(N-1)$, where $i$ is the site of slot $p$. This includes copies between
particles at the same site with different labels. The position process,
the count process and the colour processes are projections of this marked
process. Its total event rate is bounded by
\begin{equation}\label{eq:event-bound}
 N\max_{i\in I}\left(\sum_{j\ne i}q(i,j)+k_i\right).
\end{equation}
Uniformization gives a nonexplosive process. Natural motion connects all
slot position configurations, so both the position chain and the count
chain are irreducible.
The count generator on $\Omega_N$ has off-diagonal rates
\begin{equation}\label{eq:population-generator}
 L_N(c,c-e_i+e_j)=c_iq(i,j)+\frac{c_ic_jk_i}{N-1},\qquad i\ne j.
\end{equation}

\begin{lemma}\label{lem:fixation}
Every finite initial collection of labels almost surely fixes to one
label. For fixed $N$, on every compact set of positive rates with a fixed
connected motion support, the fixation time $\Theta$ satisfies
\begin{equation}\label{eq:fixation-tail}
 \sup\Prob(\Theta>r)\le(1-p)^r,\qquad r\in\N_0,
\end{equation}
for a constant $p>0$, uniformly over the initial labelled configurations
and the parameters. In particular, $\sup\E\Theta\le1/p$.
\end{lemma}

\begin{proof}
Choose a slot and prescribe that each of the other $N-1$ slots copies it
in a fixed order. After these events its label occupies every slot. On the
compact parameter set, choose a lower bound $k_*>0$ for the killing rates
and a uniformization rate $\Lambda$ greater than the bound in
\eqref{eq:event-bound}. Each prescribed copy has probability at least
$k_*/((N-1)\Lambda)$ at a tick. The probability of exactly $N-1$ ticks
in a unit interval, with these copies in the prescribed order, is at least
\[
 p=e^{-\Lambda}\frac{\Lambda^{N-1}}{(N-1)!}
 \left(\frac{k_*}{(N-1)\Lambda}\right)^{N-1}>0.
\]
Applying the Markov property successively at integer times proves
\eqref{eq:fixation-tail}. The expectation bound follows by summation.
The argument applies to any initial assignment of labels or colours.
\end{proof}

Assign a distinct label to each particle in a configuration $c$. By
Lemma~\ref{lem:fixation}, exactly one of these labels fixes, and hence
\begin{equation}\label{eq:founder-normalization}
 \sum_{i:c_i>0}c_i h_N(i,c-e_i)=1.
\end{equation}
If several labels have the same colour $A$, the corresponding fixation
events are disjoint. The probability that colour $A$ fixes is therefore
\begin{equation}\label{eq:colour-additivity}
 \mathcal F_A=\sum_{p:\,\mathrm{colour}(p)=A}
 h_N(x_p,c-e_{x_p}),
\end{equation}
where $x_p$ denotes the site of slot $p$.

Let $\mathcal S_N=\{(i,d):i\in I,\ d\in\N_0^I,\ |d|=N-1\}$. On
functions $h:\mathcal S_N\to\R$, define
\begin{equation}\label{eq:founder-operator}
\begin{aligned}
 (T_Nh)(i,d)={}&\sum_{j\ne i}q(i,j)\bigl(h(j,d)-h(i,d)\bigr)\\
 &+\sum_a d_a\sum_{b\ne a}q(a,b)
 \bigl(h(i,d-e_a+e_b)-h(i,d)\bigr)-k_i h(i,d)\\
 &+\sum_a d_ak_a\left[\frac{2}{N-1}h(i,d-e_a+e_i)\right.\\
 &\hspace{21mm}\left.+\sum_b\frac{d_b-\1_{\{a=b\}}}{N-1}
 h(i,d-e_a+e_b)-h(i,d)\right].
\end{aligned}
\end{equation}
Terms multiplied by $d_a=0$ are omitted. A term with
$d_b-\1_{\{a=b\}}=0$ is also omitted before its argument is evaluated.

Here the first two lines account for natural motion and death of the
distinguished particle. If another particle copies it, the distinguished
label has two descendants at $i$; by \eqref{eq:colour-additivity}, their
contributions add. This gives the factor $2/(N-1)$. There are
$d_b-\1_{\{a=b\}}$ possible donors at $b$ among the remaining particles.

Let $\mathcal G_N$ denote the generator of the coloured process. Summing
the contribution of each event to the colour sum of a function $h$ gives
\begin{equation}\label{eq:colour-operator}
 \mathcal G_N\!\left[\sum_{p:\,\mathrm{colour}(p)=A}
 h(x_p,c-e_{x_p})\right]
 =\sum_{p:\,\mathrm{colour}(p)=A}(T_Nh)(x_p,c-e_{x_p}).
\end{equation}
The fixation probability is harmonic. Applying
\eqref{eq:colour-operator} in a state with a single particle of colour $A$
therefore gives $T_Nh_N=0$.

\begin{proposition}\label{prop:founder-unique}
For any $i_0\in I$, the linear system
\begin{equation}\label{eq:founder-anchor}
 T_Nh=0,\qquad h(i_0,(N-1)e_{i_0})=1/N
\end{equation}
has the unique solution $h=h_N$.
\end{proposition}

\begin{proof}
Exchangeability of the slots and \eqref{eq:founder-normalization} show
that $h_N(i_0,(N-1)e_{i_0})=1/N$, so $h_N$ solves
\eqref{eq:founder-anchor}. Conversely, suppose that $h$ solves this
system and put $H(c)=\sum_i c_i h(i,c-e_i)$. With all particles coloured
$A$, equation~\eqref{eq:colour-operator} gives $L_NH=0$. Since the count
chain is finite and irreducible, $H$ is constant. At $c=Ne_{i_0}$ the
anchor gives $H(c)=1$, and thus $H=1$.

It follows from \eqref{eq:colour-operator} that the colour sum of $h$ is
harmonic in the mixed two-colour states, with boundary values zero at
extinction and one at fixation. By Lemma~\ref{lem:fixation}, the generator
$B_N$ restricted to the mixed states satisfies $e^{tB_N}\to0$ and
$(-B_N)^{-1}=\int_0^\infty e^{tB_N}\,dt$. This Dirichlet problem has a
unique solution, so the colour sum of $h$ equals the fixation probability.
Taking a state with a single particle of colour $A$ gives $h=h_N$.
\end{proof}

The same argument shows $\ker T_N=\spanof\{h_N\}$.
At constant killing, $T_N$ is an irreducible Markov generator. For general
killing its action on the constant function is
\begin{equation}\label{eq:founder-row-sum}
 (T_N\1)(i,d)=-k_i+\frac{1}{N-1}\sum_a d_ak_a.
\end{equation}
The normalization in \eqref{eq:founder-anchor} fixes the scale of its solution.

Fix a connected undirected support $E$ and weights $\nu_i>0$. Write
$q(i,j)=w_{ij}/\nu_i$ on $E$, with $w_{ij}=w_{ji}>0$, and set $q(i,j)=0$
off $E$. The conductances $w$ and killing vector $k$ range over
\begin{equation}\label{eq:parameter-domain}
 \mathcal P=(0,\infty)^E\times(0,\infty)^I.
\end{equation}

\begin{proposition}\label{prop:rationality}
For fixed $N$, the coordinates of $\pi_N$, $h_N$ and $F_{ij}^{(N)}$ are
rational functions on $\mathcal P$, each with a polynomial denominator
that has no zero in $\mathcal P$. They are real analytic there. The same
conclusion holds on the positive killing domain with the motion fixed.
\end{proposition}

\begin{proof}
The entries of $L_N$ are polynomial in $(w,k)$. Replace one row of
$L_N^{\mathsf T}$ by a row of ones. A vector in the kernel of the resulting
matrix has total mass zero. The omitted equation follows from the zero
row sums of $L_N$, so the vector also belongs to
$\ker L_N^{\mathsf T}=\spanof\{\pi_N\}$ and is therefore zero. The matrix
is invertible on $\mathcal P$, and Cramer's rule gives rational expressions
for the coordinates of $\pi_N$.

For $h_N$, consider the two-colour count chain with state
$(u,v)\in\N_0^I\times\N_0^I$, where $|u|+|v|=N$. Index its counts $z$
by categories $r$ consisting of a colour and a site. For distinct
categories $r,s$, the rate of changing one particle from $r$ to $s$ is
\begin{equation}\label{eq:two-colour-rates}
 z_r\1_{\{\mathrm{colour}(r)=\mathrm{colour}(s)\}}
 q(\mathrm{site}(r),\mathrm{site}(s))
 +\frac{z_rz_s}{N-1}k_{\mathrm{site}(r)},
\end{equation}
where we set $q(i,i)=0$ in this formula. In particular, copies between
different colours at the same site are included. Make the boundary sets
$|u|=0,N$ absorbing. The restriction $B_N$ to the mixed states and the
rate vector $r_N$ to $|u|=N$ have polynomial entries. By
Lemma~\ref{lem:fixation}, $B_N$ is invertible, and the fixation
probabilities are
\begin{equation}\label{eq:dirichlet-inverse}
 H_N=(-B_N)^{-1}r_N,
 \qquad h_N(i,d)=H_N(e_i,d).
\end{equation}
Cramer's rule gives the required rational expressions for $h_N$.
Formula~\eqref{eq:transport} gives the same conclusion for $F_{ij}^{(N)}$.
The determinant denominators have no zeros on $\mathcal P$, so these
functions are real analytic there. With the motion rates fixed, the same
argument applies to the positive killing domain.
\end{proof}

\section{Conservation and two-particle balance}\label{sec:balance}

We first show that the ancestral current has zero divergence. We then
derive the stronger pairwise balance formula for two particles.

\begin{proposition}\label{prop:conservation}
For $\varphi:I\to\R$, set
\[
 g_\varphi(c)=\sum_{i:c_i>0}c_i h_N(i,c-e_i)\varphi(i).
\]
Then
\begin{equation}\label{eq:conservation-generator}
 (L_Ng_\varphi)(c)=\sum_{i\ne j}c_iq(i,j)h_N(j,c-e_i)
 \bigl(\varphi(j)-\varphi(i)\bigr).
\end{equation}
Consequently $\sum_{j\ne i}J_{ij}^{(N)}=0$ at every site.
\end{proposition}

\begin{proof}
Give the particles distinct labels and assign to each label the weight
$\varphi$ of its present site. Keep these weights fixed during the next
event. By harmonicity and \eqref{eq:colour-additivity}, the sum of the
fixation probabilities of the old labels, multiplied by their fixed
weights, has zero infinitesimal drift.

After a copying event, the victim label is absent and both descendants of
the donor are at its old site. The fixed-weight sum is therefore equal
to $g_\varphi$ in the new configuration, including when the victim and
donor start at the same site. After a natural jump $i\to j$, the difference
between the new $g_\varphi$ and the fixed-weight sum is
$h_N(j,c-e_i)(\varphi(j)-\varphi(i))$. Multiplication by the jump rates
and summation give \eqref{eq:conservation-generator}.

Integrate \eqref{eq:conservation-generator} against $\pi_N$. The left
side vanishes, and the right side is
$-\sum_i\varphi(i)\sum_{j\ne i}J_{ij}^{(N)}$. Taking
$\varphi=\1_{\{i\}}$ proves the divergence identity at $i$.
\end{proof}

For a subset $A\subset I$, summing the divergence identity over $A$ gives
\begin{equation}\label{eq:cut-conservation}
 \sum_{i\in A,\,j\notin A}J_{ij}^{(N)}=0.
\end{equation}
Thus the current belongs to the cycle space, the space of antisymmetric
edge functions with zero divergence. If an edge is a bridge, choose $A$
to be one of the components obtained by removing it. Equation~\eqref{eq:cut-conservation}
then gives zero current on that edge. Every edge of a tree has zero
current, and on the triangle \eqref{eq:triangle} the divergence identity
gives \eqref{eq:triangle-orientation}.

For distinct sites $\alpha,\beta$, choose positive port rates $a,b$ with
\begin{equation}\label{eq:balanced-ports}
 \nu_\alpha a=\nu_\beta b.
\end{equation}
The balanced transplant flux is
\begin{equation}\label{eq:port-flux}
 \Delta_N=aF_{\alpha\beta}^{(N)}-bF_{\beta\alpha}^{(N)}.
\end{equation}
When the sites are joined by a motion edge, reversibility and
\eqref{eq:balanced-ports} give
$a/q(\alpha,\beta)=b/q(\beta,\alpha)$. Multiplying
$J_{\alpha\beta}^{(N)}$ by this common factor gives $\Delta_N$.
Balanced ports on a bridge consequently have zero flux for every $N$.
For nonadjacent sites, \eqref{eq:port-flux} uses the coefficients defined
in \eqref{eq:transport}.

For two particles we obtain a formula valid for every pair of sites.
Write
\begin{equation}\label{eq:green-definitions}
\begin{gathered}
 A=\diag(k)-Q,\qquad P_t=e^{-tA},\qquad G=A^{-1},\\
 p_t=P_t\1,\qquad \ell_t=P_tk,\qquad
 Z=\sum_x\nu_x(G\1)(x).
\end{gathered}
\end{equation}
Here $p_t(x)$ is the survival probability of one killed particle started
at $x$, and $\ell_t(x)$ is its lifetime density.

\begin{theorem}\label{thm:green-balance}
For all $i,j\in I$,
\begin{equation}\label{eq:green-flux}
 F_{ij}^{(2)}=\frac{2\nu_i}{Z}
 \int_0^\infty p_t(i)p_t(j)\,dt.
\end{equation}
Hence every balanced port pair has $\Delta_2=0$.
\end{theorem}

\begin{proof}
Let $k_- =\min_i k_i>0$. A constant killing clock of rate $k_-$ gives
$\|P_t\|_\infty\le e^{-k_-t}$. It follows that
$G=\int_0^\infty P_t\,dt$ has nonnegative entries and satisfies
$AG=GA=I$. Since $A\1=k$, we also have $Gk=\1$ and
$\ell_t=-\partial_t p_t=P_tk$.

Put $D_\nu=\diag(\nu)$. By reversibility,
$A^{\mathsf T}D_\nu=D_\nu A$, and hence $W=D_\nu G$ is symmetric.
Under the independent killed motions, the drift of the pair measure $W$
is
\begin{equation}\label{eq:green-pair-drift}
 -A^{\mathsf T}W-WA=-2D_\nu.
\end{equation}
The copying events restore mass at the diagonal pair $(x,x)$ at rate
\[
 \sum_z W(x,z)k_z+\sum_z W(z,x)k_z
 =2\nu_x(Gk)(x)=2\nu_x.
\]
This cancels \eqref{eq:green-pair-drift}. Normalizing $W$ therefore gives
the stationary probability of the ordered two-particle system:
\begin{equation}\label{eq:green-pair}
 \mu_2(x,y)=\frac{\nu_xG(x,y)}{Z}.
\end{equation}
Formula~\eqref{eq:green-pair} is the finite-state Green invariant measure
obtained by Kwa\'{s}nicki~\citep[Theorem~1.5]{Kwasnicki2024}.

Before the first death the two particles follow independent killed
motions. At that death the surviving label fixes. A label initially at
$j$, with the other label initially at $z$, consequently has fixation
probability
\begin{equation}\label{eq:two-founder-integral}
 h_2(j,e_z)=\int_0^\infty p_t(j)\ell_t(z)\,dt.
\end{equation}
By symmetry of $\mu_2$ and \eqref{eq:transport},
$F_{ij}^{(2)}=2\sum_z\mu_2(i,z)h_2(j,e_z)$; both slots contribute when
the particles start at the same site. Since $G$ commutes with $P_t$,
\[
 G\ell_t=GP_tk=P_tGk=p_t.
\]
By \eqref{eq:green-pair}, \eqref{eq:two-founder-integral} and Tonelli's
theorem,
\begin{align*}
 F_{ij}^{(2)}
 &=2\sum_z\mu_2(i,z)\int_0^\infty p_t(j)\ell_t(z)\,dt\\
 &=\frac{2\nu_i}{Z}\int_0^\infty p_t(j)(G\ell_t)(i)\,dt\\
 &=\frac{2\nu_i}{Z}\int_0^\infty p_t(i)p_t(j)\,dt.
\end{align*}
This proves \eqref{eq:green-flux}. Dividing by $\nu_i$ gives symmetry in
$i,j$, so \eqref{eq:balanced-ports} gives $\Delta_2=0$.
\end{proof}

We finish with the balance identity at constant killing, which supplies
the reference point for the perturbation in the next section.

\begin{proposition}\label{prop:neutral}
If $k_i=\kappa>0$ for all $i$, then for every $N\ge2$,
\begin{equation}\label{eq:neutral-transport}
 h_N(i,d)=\frac1N,\qquad
 \E_{\pi_N}c_i=\frac{N\nu_i}{\sum_j\nu_j},\qquad
 F_{ij}^{(N)}=\frac{\nu_i}{\sum_j\nu_j}.
\end{equation}
Every balanced port flux and every edge current is zero.
\end{proposition}

\begin{proof}
If a colour occupies $M$ slots, both its increase rate and its decrease
rate are $\kappa M(N-M)/(N-1)$. Thus $M/N$ is harmonic and has the
boundary values of the fixation probability. By
Lemma~\ref{lem:fixation}, that probability is $M/N$. Taking $M=1$ gives
$h_N=1/N$.

The copying drift of each population count cancels in pairs, so the
stationary mean vector $\rho_i=\E_{\pi_N}c_i$ satisfies $\rho Q=0$ and
$\sum_i\rho_i=N$. Irreducibility and reversibility identify
$\rho_i=N\nu_i/\sum_j\nu_j$. Substituting this and $h_N=1/N$ into
\eqref{eq:transport} gives the last identity in
\eqref{eq:neutral-transport}. The port cancellation follows from
\eqref{eq:balanced-ports}, and the edge cancellation from reversibility.
\end{proof}

At constant killing, ancestral balance also holds when the population
count chain is nonreversible, as occurs on symmetric cycle graphs with
at least four sites \citep{Corujo2021}.

\section{Circulation on the triangle}\label{sec:cubic}

We now prove the second part of Theorem~\ref{thm:main}. We keep the
motion on the triangle fixed as in \eqref{eq:triangle}. By rationality,
a nonzero current at one killing vector will suffice for each population
size.

At $N=3$ and $k=(1,2,4)$, the ten stationary count probabilities and
the eighteen founder probabilities are rational. They are listed in
Appendix~\ref{app:three}, together with the stationary equations and
the founder equation \eqref{eq:founder-operator}. Substitution
of those values in \eqref{eq:transport} gives
\begin{equation}\label{eq:three-witness}
 \Delta_3=\frac{421413160}{139117744582101}>0.
\end{equation}
The normalization and Proposition~\ref{prop:founder-unique} identify
these values as the fixation probabilities. This proves
\eqref{eq:three-witness} for the particle system.

For the rest of the calculation let $N\ge4$, put $s_0=-1,s_1=0,s_2=1$,
and set
\begin{equation}\label{eq:killing-perturbation}
 k_i(\eps)=1+\eps s_i,\qquad |\eps|<1.
\end{equation}
Write $\pi_\eps$ and $h_\eps$ for the stationary law and founder
probabilities. We first record the neutral moments and a polynomial
Poisson equation that will be used in the calculation.

Write $x=c_0$, $y=c_2$ and $c_1=N-x-y$. At $\eps=0$ the count chain has
the Dirichlet--multinomial law
\begin{equation}\label{eq:neutral-law}
 \pi_0(c)=\frac{N!}{c_0!c_1!c_2!}
 \frac{\prod_{i=0}^2(N-1)^{\overline{c_i}}}
 {[3(N-1)]^{\overline N}},
\end{equation}
where $z^{\overline a}=z(z+1)\cdots(z+a-1)$ and the empty product is one.
This is the complete-graph neutral invariant law; see also
Cloez and Thai~\citep[Theorem~2.4]{CloezThai2016Examples}. Detailed balance and normalization are established
in Appendix~\ref{app:perturb}. For falling factorials
$z_{\underline a}=z(z-1)\cdots(z-a+1)$, its moments are
\begin{equation}\label{eq:neutral-factorial-moments}
 \pi_0(x_{\underline a}y_{\underline b})
 =N_{\underline{a+b}}
 \frac{(N-1)^{\overline a}(N-1)^{\overline b}}
 {[3(N-1)]^{\overline{a+b}}},\qquad a+b\le N.
\end{equation}

Let $\mathcal P_d$ denote the polynomials in $x,y$ of total degree at
most $d$. Write $L_\eps=L_0+\eps L_1$. The neutral motion and copying
rates give
\begin{equation}\label{eq:polynomial-degrees}
 L_0\mathcal P_d\subseteq\mathcal P_d,\qquad
 L_1\mathcal P_d\subseteq\mathcal P_{d+1}.
\end{equation}
On the highest homogeneous part of degree $d\ge1$, $L_0$ acts by the
scalar
\begin{equation}\label{eq:homogeneous-eigenvalue}
 \lambda_d=-3d-\frac{d(d-1)}{N-1}.
\end{equation}
Indeed, the highest-degree motion term is
$-3(x\partial_x+y\partial_y)$. The copying terms cancel their first
derivatives in opposite pairs; their second derivatives contribute
$-(x^2\partial_{xx}+2xy\partial_{xy}+y^2\partial_{yy})/(N-1)$.
Higher finite differences have lower degree. Euler's identities give
\eqref{eq:homogeneous-eigenvalue}. For the spectrum of this neutral Moran
generator, see also Corujo~\citep{Corujo2023}.

\begin{lemma}\label{lem:poisson-descent}
For $f\in\mathcal P_d$, $d\le4$, there is a polynomial $\Psi_f$ of
degree at most $d$ satisfying
\begin{equation}\label{eq:polynomial-poisson}
 L_0\Psi_f=f-\pi_0f.
\end{equation}
It can be constructed by removing the homogeneous terms successively
from degree $d$ to degree one.
\end{lemma}

\begin{proof}
For a polynomial $r$ of degree $j\ge1$, let $[r]_j$ denote its
highest homogeneous part. Equations~\eqref{eq:polynomial-degrees} and
\eqref{eq:homogeneous-eigenvalue} show that
$L_0([r]_j/\lambda_j)$ has the same highest part. Subtracting it from
$r$ reduces the degree. Starting with $r=f-\pi_0f$ and descending from
degree $d$ to degree one constructs $\Psi_f$ of degree at most $d$.
Each subtraction preserves zero $\pi_0$-mean, so the final constant
remainder is zero.
\end{proof}

\begin{proposition}\label{prop:cubic}
For every $N\ge4$, the current for \eqref{eq:killing-perturbation}
satisfies \eqref{eq:cubic-intro}.
\end{proposition}

\begin{proof}
\emph{Step 1: the fixation probabilities.}
Proposition~\ref{prop:rationality} gives convergent expansions around zero,
\begin{equation}\label{eq:taylor-expansions}
 \pi_\eps=\sum_{r\ge0}\eps^r\pi_r,
 \qquad h_\eps=\sum_{r\ge0}\eps^r h^{[r]}.
\end{equation}
These are Taylor coefficients without factorials. The population size
$N$ stays fixed throughout the expansion. In particular, $h^{[0]}=1/N$.

Write $T_\eps=T_0+\eps T_1$ for the founder operator. Constant killing
makes $T_0$ an irreducible Markov generator on $\mathcal S_N$. Comparing
coefficients in $T_\eps h_\eps=0$ gives
\begin{equation}\label{eq:founder-recursion}
 T_0h^{[r]}=-T_1h^{[r-1]},\qquad
 h^{[r]}(0,(N-1)e_0)=0\quad(r\ge1).
\end{equation}
Its solution is unique because $\ker T_0$ consists of constants.

Put $S=d_2-d_0$ and $R=d_0+d_2$. The first coefficient is
\begin{equation}\label{eq:first-founder}
 h^{[1]}(i,d)=\frac{S-(N-1)s_i}{N(3N-1)}.
\end{equation}
The next two coefficients have the polynomial forms
\begin{equation}\label{eq:founder-polynomial-forms}
\begin{aligned}
 h^{[2]}(i,d)&=a_2S^2+b_2R+c_2s_iS+\delta_2s_i^2,\\
 h^{[3]}(i,d)&=\beta_0S^3+\beta_1SR+\beta_2S+\beta_3s_iS^2
 +\beta_4s_iR+\beta_5s_i+\beta_6s_i^2S.
\end{aligned}
\end{equation}
All coefficient values, the shift formula for $T_0,T_1$, and the
identities establishing \eqref{eq:founder-recursion} are given in
Appendix~\ref{app:perturb}. The coefficients are rational functions of
$N$. Substitution verifies the recursion and its normalization;
uniqueness then identifies these polynomials as the Taylor coefficients.

\emph{Step 2: the stationary law.}
The derivatives of $\pi_\eps$ can be computed from polynomial Poisson
equations under the neutral law.
Define $\Phi_r(f)=\pi_rf$ and $\mathcal R f=-L_1\Psi_f$. Normalization
gives $\Phi_r(1)=0$ for $r\ge1$. The stationary equation yields
$\pi_rL_0=-\pi_{r-1}L_1$. For $r\ge1$, we obtain
\begin{equation}\label{eq:stationary-recursion}
\begin{aligned}
 \Phi_r(f)&=\pi_r(f-\pi_0f)=\pi_rL_0\Psi_f\\
 &=-\pi_{r-1}L_1\Psi_f=\Phi_{r-1}(\mathcal Rf),\\
 \Phi_r(f)&=\Phi_0(\mathcal R^rf)
\end{aligned}
\end{equation}
whenever the indicated polynomial Poisson problems have degree at most
four. Adding a constant to $\Psi_f$ has no effect because $L_1$ kills
constants. Each application of $\mathcal R$ raises the degree by at most
one. Thus a third-order stationary response of a linear observable uses
only fourth moments.

\emph{Step 3: the current.}
Define
\begin{equation}\label{eq:flux-polynomials}
 G_r(x,y)=x h^{[r]}(2,c-e_0)-y h^{[r]}(0,c-e_2),
 \qquad c=(x,N-x-y,y),\quad r=0,1,2,3.
\end{equation}
Terms with a zero leading count are evaluated as zero. Polynomial
extension gives the same value at those boundary states. We have
$\deg G_r\le r+1$, and expansion of the current gives
\begin{equation}\label{eq:cubic-four-terms}
 [\eps^3]\Delta_N
 =\Phi_0G_3+\Phi_1G_2+\Phi_2G_1+\Phi_3G_0.
\end{equation}
Since $\deg G_r\le r+1$ and $\mathcal R$ raises the degree by at most
one, each term in \eqref{eq:cubic-four-terms} is determined by neutral
moments of degree at most four. Substitution gives the four values in
\eqref{eq:app-four-contributions}. The Python script in
Supplementary materials reproduces these values by exact rational
arithmetic. Their sum factors as
\begin{equation}\label{eq:cubic-coefficient}
 [\eps^3]\Delta_N
 =-\frac{2(N-3)(N-2)(N-1)^2}
 {9N(3N-2)^2(3N-1)^2(3N+1)}.
\end{equation}

The cancellation of the first-order response can be seen directly. From
\eqref{eq:first-founder},
\begin{equation}\label{eq:first-flux-polynomials}
 G_0=\frac{x-y}{N},\qquad
 G_1=-\frac{(x-y)^2+(N-2)(x+y)}{N(3N-1)}.
\end{equation}
The neutral moments give
\begin{equation}\label{eq:first-order-cancellation}
 \Phi_0G_1=-\frac{2(N-1)}{3(3N-2)},\qquad
 \Phi_1G_0=\frac{2(N-1)}{3(3N-2)}.
\end{equation}
For the second identity, use
$\Psi_{G_0}=-(x-y)/(3N)$ and
\[
 \mathcal RG_0=\frac{N(x+y)-(x-y)^2}{3N(N-1)}.
\]
Thus the first change in the founder weights is exactly offset by the
first change in the stationary law.

The reflection $0\leftrightarrow2$ changes $k(\eps)$ into $k(-\eps)$
and reverses the port pair. Uniqueness of the stationary law and of the
fixation probabilities gives $\Delta_N(-\eps)=-\Delta_N(\eps)$.
The current is analytic, so its even coefficients vanish. Together with
\eqref{eq:cubic-coefficient} and \eqref{eq:first-order-cancellation},
this proves \eqref{eq:cubic-intro}. For each fixed $N\ge4$, sufficiently
small positive $\eps$ gives strictly negative current.
\end{proof}

The factor $N-3$ in \eqref{eq:cubic-coefficient} makes the cubic term
vanish at $N=3$ for this perturbation. The killing vector $(1,2,4)$
provides the nonzero three-particle value in \eqref{eq:three-witness}.
We finish by passing from these values to generic circulation.

\begin{lemma}\label{lem:algebraic}
On a positive parameter domain as in \eqref{eq:parameter-domain}, a
balanced port flux is either identically zero or has a zero set equal
to the intersection of the domain with a proper algebraic set. In the
second case its zero set has Lebesgue measure zero and empty interior,
and its nonzero set is open and dense.
\end{lemma}

\begin{proof}
By Proposition~\ref{prop:rationality}, write the flux as $p/d$, where
$p,d$ are polynomials and $d$ has no zero in the domain. A nonzero flux
at one point makes $p$ a nonzero polynomial.

For completeness, a nonzero polynomial has a null zero set. This is
immediate in one variable. In $r$ variables write
$p(x',t)=\sum_j p_j(x')t^j$ and choose a nonzero coefficient $p_j$.
Induction makes its zero set in $r-1$ variables null. Outside that set,
the polynomial in $t$ has finitely many roots. Fubini's theorem on
bounded boxes, followed by a countable union of boxes, proves the claim.
No nonempty open ball is null, so the zero set has empty interior.
Continuity gives openness of the nonzero set.
\end{proof}

\begin{proof}[Proof of Theorem~\ref{thm:main}]
Conservation is Proposition~\ref{prop:conservation}, and two-particle
balance is Theorem~\ref{thm:green-balance}. On the fixed triangle,
\eqref{eq:three-witness} gives a nonzero killing vector for $N=3$.
For each $N\ge4$, \eqref{eq:cubic-intro} gives one with positive
coordinates. Apply Lemma~\ref{lem:algebraic} on the killing domain for
each fixed $N$. This proves all assertions about $\mathcal U_N$.

The complement of $\mathcal U_*$ is the countable union of the null
complements of $\mathcal U_N$, so it is null. The set $\mathcal U_*$ is
a countable intersection of open sets, hence a $G_\delta$.
Every open ball in $(0,\infty)^3$ has positive measure and intersects
$\mathcal U_*$. This proves density and supplies a single killing
vector valid for every $N\ge3$.
\end{proof}

\begin{corollary}\label{cor:persistence}
Fix a finite set $\mathcal A\subset\{3,4,\ldots\}$ and a killing vector
$k\in\bigcap_{N\in\mathcal A}\mathcal U_N$. There is a neighbourhood of
the corresponding conductances and killing rates on which every
$\Delta_N$, $N\in\mathcal A$, has the same sign as at that vector and
\begin{equation}\label{eq:half-margin}
 |\Delta_N(\vartheta)|\ge\tfrac12|\Delta_N(\vartheta_*)|.
\end{equation}
Here $\vartheta=(w,k)$ and $\vartheta_*$ is the initial parameter point.
\end{corollary}

\begin{proof}
Use continuity in Proposition~\ref{prop:rationality} and intersect the
finitely many resulting neighbourhoods.
\end{proof}

\section{Rare successful transplants}\label{sec:transplants}

We give a probabilistic interpretation of the transport coefficient.
Run the uncoloured system until a rare proposal moves a particle between
two specified sites. Then follow the moved particle's label until
fixation. The probability that the proposal occurs in one direction and
its label fixes has a simple limit as proposals become rare.

Fix the cell, $N$, and balanced ports $\alpha,\beta,a,b$ as in
\eqref{eq:balanced-ports}. Let
\begin{equation}\label{eq:proposal-fraction}
 \theta=\frac{a}{a+b},\qquad
 v_\theta(c)=\theta c_\alpha+(1-\theta)c_\beta.
\end{equation}
During the uncoloured evolution, each particle at $\alpha$ proposes a
move to $\beta$ at rate $\lambda\theta$, and each particle at $\beta$
proposes a move to $\alpha$ at rate $\lambda(1-\theta)$.
At the first proposal, move that particle, assign it mutant colour and
assign all other particles resident colour. Switch off proposals and run
the coloured system until fixation. Write $S_\lambda^+(c)$ and
$S_\lambda^-(c)$ for the probabilities of mutant fixation following a
proposal in the respective directions, including the probability of
that proposal direction, from initial counts $c$.

\begin{proposition}\label{prop:rare-success}
As $\lambda\downarrow0$,
\begin{equation}\label{eq:rare-success}
 S_\lambda^+(c)-S_\lambda^-(c)
 =\frac{\Delta_N}
 {a\E_{\pi_N}c_\alpha+b\E_{\pi_N}c_\beta}+O(\lambda),
\end{equation}
uniformly in $c\in\Omega_N$. A nonzero balanced flux therefore
gives a strict directional bias for sufficiently rare proposals.
\end{proposition}

\begin{proof}
Prior to the proposal the count chain is killed at rate
$\lambda v_\theta$. Its occupation resolvent is
\[
 R_{\lambda,\theta}=(\lambda M_{v_\theta}-L_N)^{-1},
\]
where $M_v$ denotes multiplication by $v$. The inverse exists because
the finite irreducible chain reaches states with positive proposal rate.
The density of a proposal and the subsequent fixation probability give
\begin{equation}\label{eq:proposal-resolvent}
 S_\lambda^\pm=\lambda R_{\lambda,\theta}g_\pm,
\end{equation}
where
\[
 g_+(c)=\theta c_\alpha h_N(\beta,c-e_\alpha),\qquad
 g_-(c)=(1-\theta)c_\beta h_N(\alpha,c-e_\beta).
\]
As throughout, a term with zero leading count is zero.
Lemma~\ref{lem:sparse-resolvent} gives
\[
 \lambda R_{\lambda,\theta}g_\pm
 =\frac{\pi_Ng_\pm}{\pi_Nv_\theta}\1+O(\lambda)
\]
uniformly in $c$. Inserting \eqref{eq:transport} and
\eqref{eq:proposal-fraction} proves \eqref{eq:rare-success}.
\end{proof}

\begin{remark}\label{rem:colour-speeds}
The bias persists under a small change in the mutant clock.
Multiply the natural motion and killing rates of every mutant
particle by $t>0$, keep the resident rates unchanged, and let $f_t(j,d)$
be mutant fixation probability from one mutant at $j$ and resident
counts $d$. Copying inherits the donor's position and colour, including
the colour's speed. The finite absorption matrix is invertible for
every $t>0$, so
$f_t$ is rational and continuous in $t$, with $f_1=h_N$.
On each compact speed interval, fixation has a uniform exponential tail
by the marked copying argument in Lemma~\ref{lem:fixation}.

Define
\begin{equation}\label{eq:speed-port-coefficients}
\begin{aligned}
 U_N(t)&=a\sum_{c:c_\alpha>0}\pi_N(c)c_\alpha
 f_t(\beta,c-e_\alpha),\\
 D_N(t)&=b\sum_{c:c_\beta>0}\pi_N(c)c_\beta
 f_t(\alpha,c-e_\beta).
\end{aligned}
\end{equation}
If necessary, orient the port pair so that $\Delta_N>0$. Since
$U_N(1)-D_N(1)=\Delta_N$, continuity gives a compact interval
$K\subset(0,\infty)$ containing $1$ in its interior such that
\begin{equation}\label{eq:speed-bias}
 U_N(t_+)-D_N(t_-)\ge\tfrac34\Delta_N
 \qquad(t_+,t_-\in K).
\end{equation}
This includes a common mutant speed $t_+=t_-=t$, as well as reciprocal
speeds $t_+=r$, $t_-=r^{-1}$ for $r>1$ sufficiently close to one.
When two colour clocks are exchanged, normalization by the resident
clock gives reciprocal mutant speeds.
The proof of Proposition~\ref{prop:rare-success} applies with $f_{t_+}$
and $f_{t_-}$ in the two directions. Its leading imbalance is
$[U_N(t_+)-D_N(t_-)]/[a\E_{\pi_N}c_\alpha+b\E_{\pi_N}c_\beta]$.
The remainder is uniform over $K^2$, since the fixation probabilities
are bounded by one. Appendix~\ref{app:estimates} gives the uniform
resolvent bound and fixation-time estimate.
\end{remark}

\section{Relation to the multiparticle nonexplosion problem}\label{sec:discussion}

Proposition~\ref{prop:rare-success} turns ancestral circulation into a
directional bias for successful migration. This gives a route from the
finite transport theory to the multiparticle nonexplosion question posed by
Kwa\'{s}nicki~\citep[Remark~1.2]{Kwasnicki2024}. Let $X$ be a killed
Hunt process on $D$ with finite reference measure $m$. Write $P_t$ for
its sub-Markov semigroup, and suppose that it is self-dual:
\begin{equation}\label{eq:hunt-duality}
 \int_D fP_tg\,dm=\int_D gP_tf\,dm
 \qquad(t\ge0)
\end{equation}
for bounded nonnegative $f,g$. Write $\Prob^{\boldsymbol x}$ for the law
of the $N$-particle system started from $\boldsymbol x\in D^N$. Let
$\tau_n^{(N)}$ be the successive branching times and put
$\tau_\infty^{(N)}=\lim_{n\to\infty}\tau_n^{(N)}$. The question is
whether self-duality and finite reference mass imply
\begin{equation}\label{eq:multiparticle-question}
\begin{gathered}
 \Prob^{\boldsymbol x}(\tau_\infty^{(N)}=\infty)=1,\\
 \text{for }m^{\otimes N}\text{-almost every }
 \boldsymbol x\in D^N,\qquad N\ge3.
\end{gathered}
\end{equation}
Kwa\'{s}nicki's theorem establishes the two-particle case under these
symmetry and mass assumptions.

For larger populations, nonexplosion has been proved for Brownian motion
under geometric conditions on the domain
\citep{BieniekBurdzyFinch2012,GrigorescuKang2012} and for diffusions
under conditions controlling their drift and rebirth near the boundary
\citep{Villemonais2011,Villemonais2014}. Strong drift towards the killing
boundary can instead produce branching explosion
\citep{BieniekBurdzyPal2012}. The question in
\eqref{eq:multiparticle-question} asks what follows from symmetry and
finite reference mass alone.

The finite theory gives a stable directional mechanism for every
prescribed $N\ge3$. The speed perturbation \eqref{eq:speed-bias}
preserves the success imbalance, while
Proposition~\ref{prop:rare-success} and Appendix~\ref{app:estimates}
control rare proposals and the duration of fixation.

We are actively extending this method to self-dual Hunt processes with
finite reference measure through an infinite sequence of reversible
cells. The general construction and its implications for the
multiparticle nonexplosion problem will be presented in subsequent work.

\section*{Statements and Declarations}

\paragraph{Funding.}
The authors received no funding for this research.

\paragraph{Competing interests.}
The authors declare that they have no competing interests.

\paragraph{Data and code availability.}
Lean proof sources for the principal results and a Python script for
exact symbolic verification are provided in
Supplementary materials.

\paragraph{Ethics.}
The work is mathematical and involves no human participants or animal
subjects.

\paragraph{Use of AI tools.}
Parts of the Lean formalization were developed with assistance from
Tex2Lean. All formalized results used in this work were reviewed and
verified by the authors.
% The authors confirmed no funding and no competing interests.
% The user requested that author contributions be left unfilled.
% Submission-field requirements are recorded in notes/submission_information.md.

\bibliographystyle{spmpsci}
\bibliography{references}
\clearpage
\appendix
\section*{Appendices}
\addcontentsline{toc}{section}{Appendices}
The appendices give the explicit three-particle example, the full cubic
response calculation, and the finite estimates for rare transplants.
\section{The three-particle example}
\label{app:three}

Consider the complete graph on $I=\{0,1,2\}$ with
\[
q(i,j)=1\quad(i\ne j),\qquad
\nu_i=1,\qquad k=(1,2,4),\qquad N=3.
\]
The two ports are $0$ and $2$, with unit port rates.  The population
generator on $\Omega_3=\{c\in\mathbb N_0^3:|c|=3\}$ is
\begin{equation}
 L(c,c-e_i+e_j)=c_i+\frac{c_ic_jk_i}{2},\qquad i\ne j,
 \label{eq:app-three-generator}
\end{equation}
and its diagonal entries are minus the corresponding row sums.
Here $e_i$ denotes the unit vector at $i$.

\subsection{Stationary population weights}

We first determine the stationary population law.  Let
$Z_3=2121214$, and let $W(c)$ and $B(c)$ have the values in
Table~\ref{tab:app-three-stationary}.
\begin{table}[htbp]
\centering
\begin{tabular}{rrr}
\hline
$c$ & $W(c)$ & $B(c)$ \\
\hline
$(0,0,3)$ & $103091$ & $618546$ \\
$(0,1,2)$ & $113786$ & $1365432$ \\
$(0,2,1)$ & $157674$ & $1892088$ \\
$(0,3,0)$ & $219448$ & $1316688$ \\
$(1,0,2)$ & $138594$ & $1524534$ \\
$(1,1,1)$ & $191246$ & $2486198$ \\
$(1,2,0)$ & $264159$ & $2377431$ \\
$(2,0,1)$ & $239661$ & $2636271$ \\
$(2,1,0)$ & $329225$ & $2963025$ \\
$(3,0,0)$ & $364330$ & $2185980$ \\
\hline
\end{tabular}
\caption{Stationary weights and the corresponding outgoing mass.}
\label{tab:app-three-stationary}
\end{table}

For every $c\in\Omega_3$, these values satisfy
\begin{equation}
\begin{split}
 \sum_{c\in\Omega_3}W(c)&=Z_3,\\
 W(c)\sum_{i\ne j}\left(c_i+\frac{c_ic_jk_i}{2}\right)
 &=B(c)\\
 &=\sum_{\substack{i\ne j\\c_j>0}}
 W(c+e_i-e_j)
 \left[(c_i+1)+\frac{(c_i+1)(c_j-1)k_i}{2}\right].
\end{split}
\label{eq:app-three-balance}
\end{equation}
The last sum is the incoming mass at $c$.  For example, when
$c=(3,0,0)$, it is
$3\cdot329225+5\cdot239661=2185980=6\cdot364330$.
Substitution gives the remaining values of $B(c)$ as well.
Consequently $WL=0$ and $\sum_cW(c)=Z_3$.  By irreducibility,
the stationary law is
\begin{equation}
 \pi_3(c)=\frac{W(c)}{Z_3}.
 \label{eq:cert-pi3}
\end{equation}

\subsection{The founder equations and their solution}

Let $h_3(i,d)$ be the fixation probability of a specified founder at
$i$, when the other two particles have count vector $d$.  Put
$H_3=262336086$ and write
$u_i^{ab}=H_3h_3(i,e_a+e_b)$, with $u_i^{ab}=u_i^{ba}$.
The eighteen entries are listed in Table~\ref{tab:app-three-founders}.
\begin{table}[htbp]
\centering
\setlength{\tabcolsep}{4pt}
\small
\begin{tabular}{rrrrrrr}
\hline
$i\backslash d$ & $(0,0,2)$ & $(0,1,1)$ & $(0,2,0)$
 & $(1,0,1)$ & $(1,1,0)$ & $(2,0,0)$ \\
\hline
$0$ & $138983140$ & $121868982$ & $108221552$
 & $108047078$ & $96896756$ & $87445362$ \\
$1$ & $117160434$ & $100452990$ & $87445362$
 & $87483403$ & $77057267$ & $68542574$ \\
$2$ & $87445362$ & $72587826$ & $61430106$
 & $61676473$ & $52983701$ & $46241930$ \\
\hline
\end{tabular}
\caption{The integer values $H_3h_3(i,d)$.}
\label{tab:app-three-founders}
\end{table}

To determine $h_3$, we write its harmonic equations.  For a function
$h$ on $\{(i,d):|d|=2\}$, define
\begin{equation}
\begin{split}
 (Th)(i,d)={}&\sum_{j\ne i}[h(j,d)-h(i,d)]\\
 &+\sum_a d_a\sum_{b\ne a}
 [h(i,d-e_a+e_b)-h(i,d)]-k_ih(i,d)\\
 &+\sum_a d_ak_a\left[
 h(i,d-e_a+e_i)
 +\frac12\sum_b(d_b-\1_{\{a=b\}})h(i,d-e_a+e_b)
 -h(i,d)\right].
\end{split}
\label{eq:app-three-founder-operator}
\end{equation}
Terms with $d_a=0$ vanish.  When a particle at $a$ dies, the founder
is the donor with probability $1/2$.  Its two descendants account
for the first term in the brackets; the other donor accounts for
the second.  Both terms include copying at the same site.

Thus $Tu=0$ consists of the following six equations for each
$i\in\{0,1,2\}$:
\begin{equation}
\begin{aligned}
0={}&\sum_{j\ne i}u_j^{22}
 +2(u_i^{02}+u_i^{12})+8u_i^{i2}-(10+k_i)u_i^{22},\\
0={}&2\sum_{j\ne i}u_j^{12}
 +6u_i^{11}+4u_i^{22}+2u_i^{01}+2u_i^{02}
 +4u_i^{i2}+8u_i^{i1}-2(12+k_i)u_i^{12},\\
0={}&\sum_{j\ne i}u_j^{11}
 +2(u_i^{01}+u_i^{12})+4u_i^{i1}-(8+k_i)u_i^{11},\\
0={}&2\sum_{j\ne i}u_j^{02}
 +6u_i^{00}+3u_i^{22}+2u_i^{01}+2u_i^{12}
 +2u_i^{i2}+8u_i^{i0}-2(11+k_i)u_i^{02},\\
0={}&2\sum_{j\ne i}u_j^{01}
 +4u_i^{00}+3u_i^{11}+2u_i^{02}+2u_i^{12}
 +2u_i^{i1}+4u_i^{i0}-2(9+k_i)u_i^{01},\\
0={}&\sum_{j\ne i}u_j^{00}
 +2(u_i^{01}+u_i^{02})+2u_i^{i0}-(7+k_i)u_i^{00}.
\end{aligned}
\label{eq:app-three-eighteen-equations}
\end{equation}
We multiplied the equations for mixed $d$ by two.  The entries in
Table~\ref{tab:app-three-founders} satisfy all eighteen equations.
For example, the last one with $i=0$ reads
\[
 68542574+46241930+2(96896756+108047078)
 +2\cdot87445362=8\cdot87445362.
\]
The normalization identities are
\begin{equation}
 \sum_{i:c_i>0}c_i u_i^{\,c-e_i}=H_3,
 \qquad c\in\Omega_3,
 \label{eq:app-three-normalization}
\end{equation}
where $u_i^d$ denotes the table entry with column $d$.
Substitution verifies \eqref{eq:app-three-normalization} at each
of the ten configurations.  In particular, $u_0^{00}=H_3/3$.

It remains to identify this normalized solution with the fixation
probabilities.  Assign distinct ancestral labels to the particles,
and sum $u_i^d/H_3$ over those of a chosen colour.  By
\eqref{eq:app-three-founder-operator}, the resulting function is
harmonic for the coloured process.  Its boundary values are zero
at extinction and one at fixation, by
\eqref{eq:app-three-normalization}.  From every mixed state, two
specified deaths with the same surviving donor have positive
probability in a fixed time interval.  Since there are finitely
many states, the time to extinction or fixation has an exponential
tail.  Uniqueness of the harmonic extension now gives
\begin{equation}
 h_3(i,d)=\frac{u_i^d}{H_3}.
 \label{eq:cert-h3}
\end{equation}

\subsection{The resulting transport}

The stationary transplant coefficients are
\[
 F_{ij}^{(3)}=\sum_{c:c_i>0}\pi_3(c)c_i h_3(j,c-e_i).
\]
Using \eqref{eq:cert-pi3} and \eqref{eq:cert-h3}, we obtain
\begin{equation}
 \setlength{\arraycolsep}{4pt}
 \bigl(F_{ij}^{(3)}\bigr)_{i\ne j}
 =\frac{1}{278235489164202}
 \begin{pmatrix}
 *&113067935100535&78610412314267\\
 113068777926855&*&64147123261485\\
 78609569487947&64147966087805&*
 \end{pmatrix}.
 \label{eq:app-three-flux-matrix}
\end{equation}
The stars denote the diagonal entries, which do not enter the
calculation.  Each displayed entry is the sum of at most six terms,
with denominator $Z_3H_3=556470978328404$ before simplification.
In particular,
\begin{equation}
 \Delta_3=F_{02}^{(3)}-F_{20}^{(3)}
 =\frac{421413160}{139117744582101}>0.
 \label{eq:app-three-nonzero}
\end{equation}
The other two differences give
$J_{02}^{(3)}=J_{21}^{(3)}=J_{10}^{(3)}=\Delta_3$.

\clearpage
\section{The cubic transport coefficient for every \texorpdfstring{$N\ge4$}{N >= 4}}
\label{app:perturb}

Fix $N\ge4$, let $q(i,j)=1$ for $i\ne j$ on $I=\{0,1,2\}$,
and use ports $0$ and $2$ with unit port rates.  Throughout this
appendix,
\begin{equation}
 s_0=-1,\qquad s_1=0,\qquad s_2=1,\qquad
 k_i(\varepsilon)=1+\varepsilon s_i,\qquad |\varepsilon|<1.
 \label{eq:app-perturbation-family}
\end{equation}
Write $\pi_\varepsilon$ for the stationary population law and
$h_\varepsilon(i,d)$ for the probability that a specified founder
at $i$ fixes, given the other $N-1$ particles' count vector $d$.
Analyticity of these finite systems gives convergent expansions
near zero:
\begin{equation}
 \pi_\varepsilon=\sum_{r\ge0}\varepsilon^r\pi_r,
 \qquad
 h_\varepsilon=\sum_{r\ge0}\varepsilon^rh^{[r]},
 \qquad h^{[0]}=\frac1N.
 \label{eq:app-taylor-expansions}
\end{equation}
The bracketed superscript denotes Taylor order.  We determine the
founder coefficients through order three and use the stationary
equations to compute the corresponding transport coefficient.

\subsection{Founder coefficients}

For $|d|=N-1$, set
\[
 S=d_2-d_0,\qquad R=d_0+d_2.
\]
The first three coefficients are
\begin{equation}
\begin{aligned}
 h^{[1]}(i,d)&=\frac{S-(N-1)s_i}{N(3N-1)},\\
 h^{[2]}(i,d)&=a_2S^2+b_2R+c_2s_iS+\delta_2s_i^2,\\
 h^{[3]}(i,d)&=\beta_0S^3+\beta_1SR+\beta_2S+\beta_3s_iS^2\\
 &\quad+\beta_4s_iR+\beta_5s_i+\beta_6s_i^2S.
\end{aligned}
\label{eq:app-founder-coefficients}
\end{equation}
The four second order coefficients are
\begin{equation}
\begin{aligned}
 a_2&=\frac{1}{3N^2(3N-1)},&
 b_2&=\frac{-3N^2+2N-3}{3N^2(3N-1)^2},\\
 c_2&=\frac{2-N}{3N^2(3N-1)},&
 \delta_2&=\frac{3N^3-8N^2+9N-4}{3N^2(3N-1)^2}.
\end{aligned}
\label{eq:app-second-founder-coefficients}
\end{equation}
The seven third order coefficients are
\begin{equation}
\begin{aligned}
 \beta_0&=\frac{1}{3N^2(3N-1)(3N+1)},\\
 \beta_1&=-\frac{6N^3-4N^2+9N-1}
 {3N^3(3N-1)^2(3N+1)},\\
 \beta_2&=\frac{9N^6+21N^5-28N^4+30N^3+5N^2-15N+2}
 {3N^3(3N-2)(3N-1)^3(3N+1)},\\
 \beta_3&=-\frac{N-3}{3N^2(3N-1)(3N+1)},\\
 \beta_4&=\frac{9N^5-24N^4+23N^3-32N^2+22N-2}
 {3N^3(3N-2)(3N-1)^2(3N+1)},\\
 \beta_5&=-\frac{(N-1)(9N^6-33N^5+89N^4-135N^3+118N^2-44N+4)}
 {3N^3(3N-2)(3N-1)^3(3N+1)},\\
 \beta_6&=\frac{9N^5-54N^4+105N^3-90N^2+28N-2}
 {3N^3(3N-2)(3N-1)^2(3N+1)}.
\end{aligned}
\label{eq:app-third-founder-coefficients}
\end{equation}

To prove these formulas, let $\kappa=(\kappa_0,\kappa_1,\kappa_2)$
and write the founder operator on triples of polynomials
$H=(H_0,H_1,H_2)$ as
\begin{equation}
\begin{split}
 (\mathsf T_\kappa H)_i(S,R)={}&\sum_{j\ne i}[H_j(S,R)-H_i(S,R)]\\
 &+\sum_a d_a\sum_{b\ne a}
 [H_i(S+s_b-s_a,R+s_b^2-s_a^2)-H_i(S,R)]
 -\kappa_iH_i(S,R)\\
 &+\sum_a d_a\kappa_a\left[
 \frac{2}{N-1}H_i(S+s_i-s_a,R+s_i^2-s_a^2)\right.\\
 &\hspace{19mm}\left.
 +\sum_b\frac{d_b-\1_{\{a=b\}}}{N-1}
 H_i(S+s_b-s_a,R+s_b^2-s_a^2)-H_i(S,R)
 \right],
\end{split}
\label{eq:app-founder-shift}
\end{equation}
where
\begin{equation}
 d_0=\frac{R-S}{2},\qquad d_1=N-1-R,\qquad d_2=\frac{R+S}{2}.
 \label{eq:app-founder-count-coordinates}
\end{equation}
The six ordered changes of another particle have shifts
\[
\begin{array}{c|rrrrrr}
 a\to b&0\to1&0\to2&1\to0&1\to2&2\to0&2\to1\\
\hline
 (\delta_S,\delta_R)&(1,-1)&(2,0)&(-1,1)&(1,1)&(-2,0)&(-1,-1).
\end{array}
\]
Copying the founder gives the same shift with $b=i$.

Set
\[
 T_0=\mathsf T_{(1,1,1)},\qquad
 T_1=\mathsf T_{(-1,0,1)}-\mathsf T_{(0,0,0)}.
\]
Thus $T_\varepsilon=T_0+\varepsilon T_1$.  Comparing Taylor
coefficients in the founder equation gives
\begin{equation}
 T_0h^{[r]}+T_1h^{[r-1]}=0,
 \qquad h^{[r]}(0,(N-1)e_0)=0\quad(r\ge1).
 \label{eq:app-founder-recursion}
\end{equation}
The normalization follows from exchangeability of the $N$ labels when
all particles occupy one site.  At neutrality,
$T_01=0$, and every off diagonal entry of $T_0$ is nonnegative.
The natural motion terms connect all pairs $(i,d)$, so $T_0$ is
an irreducible finite generator.  Its kernel consists of constants.
It follows that each equation in \eqref{eq:app-founder-recursion},
with the indicated normalization, has at most one solution.

We verify \eqref{eq:app-founder-recursion} for the expressions in
\eqref{eq:app-founder-coefficients}.  Work over $\mathbb Q(N)$ and
expand the shifts by
\begin{equation}
 (S+u)^a(R+v)^b
 =\sum_{p=0}^a\sum_{q=0}^b
 \binom ap\binom bq S^pR^q u^{a-p}v^{b-q}.
 \label{eq:app-binomial-shifts}
\end{equation}
Substituting \eqref{eq:app-founder-coefficients} and
\eqref{eq:app-founder-count-coordinates} into
\eqref{eq:app-founder-shift} gives, for nonnegative $a,b$,
\begin{equation}
\begin{gathered}
 \bigl([S^aR^b](T_0h^{[r]}+T_1h^{[r-1]})_i\bigr)_{i=0}^2
 =(0,0,0),\\
 a+b\le r,\qquad r=1,2,3.
\end{gathered}
\label{eq:app-founder-residuals}
\end{equation}
There are no terms of higher degree.  Indeed, multiply the
coefficients obtained from \eqref{eq:app-binomial-shifts} by
\[
 D(N)=12N^3(N-1)(3N-2)(3N-1)^3(3N+1).
\]
For each $(a,b)$ in \eqref{eq:app-founder-residuals} and each $i$,
the resulting polynomial in $N$ has all coefficients zero.
For the first order terms this calculation reads
\begin{equation}
\begin{aligned}
 T_1h^{[0]}&=\frac{S-(N-1)s_i}{N(N-1)},\\
 T_0h^{[1]}&=-\frac{S-(N-1)s_i}{N(N-1)}.
\end{aligned}
 \label{eq:app-first-founder-check}
\end{equation}
For orders two and three, the same expansion involves monomials
of degree at most two and three, respectively.

The normalizations reduce to the scalar identities
\begin{equation}
\begin{aligned}
 a_2(N-1)^2+(b_2+c_2)(N-1)+\delta_2&=0,\\
 \beta_0(N-1)^3+(\beta_1+\beta_3)(N-1)^2&\\
 \quad +(\beta_2+\beta_4+\beta_6)(N-1)+\beta_5&=0.
\end{aligned}
\label{eq:app-founder-anchors}
\end{equation}
Substitution of the displayed fractions proves both identities.
The first order normalization follows at once from the formula for $h^{[1]}$.
Equations~\eqref{eq:app-founder-residuals} and
\eqref{eq:app-founder-anchors}, together with uniqueness of the
normalized equations, prove \eqref{eq:app-founder-coefficients}.
The founder normalization further gives
\begin{equation}
 \sum_{i:c_i>0}c_i h^{[r]}(i,c-e_i)=0,
 \qquad c\in\Omega_N,\quad r=1,2,3.
 \label{eq:app-taylor-founder-normalization}
\end{equation}

\subsection{The neutral population law}

For an integer $a\ge0$, define the rising and falling factorials
\[
 t^{\overline a}=\prod_{j=0}^{a-1}(t+j),\qquad
 (t)_{\underline a}=\prod_{j=0}^{a-1}(t-j),
\]
with empty products equal to one.  Set $\alpha=N-1$.
The neutral population law is
\begin{equation}
 \pi_0(c)=\frac{N!}{c_0!c_1!c_2!}
 \frac{\prod_{i=0}^2\alpha^{\overline{c_i}}}
 {(3\alpha)^{\overline N}},\qquad |c|=N.
 \label{eq:app-neutral-law}
\end{equation}
For positive integer parameters $\alpha_0,\alpha_1,\alpha_2$,
the identity
\begin{equation}
 \sum_{|c|=M}\frac{M!}{c_0!c_1!c_2!}
 \prod_{i=0}^2\alpha_i^{\overline{c_i}}
 =(\alpha_0+\alpha_1+\alpha_2)^{\overline M}
 \label{eq:app-rising-convolution}
\end{equation}
follows by comparing coefficients of $z^M$ in
$\prod_i(1-z)^{-\alpha_i}$.
Here
$(1-z)^{-\alpha_i}=\sum_{r\ge0}\alpha_i^{\overline r}z^r/r!$,
as follows by multiplying $\alpha_i$ geometric series.  Taking
$M=N$ and all three parameters equal to $\alpha$ shows that
\eqref{eq:app-neutral-law} is a probability measure.
For $c'=c-e_i+e_j$,
\begin{equation}
 \frac{\pi_0(c')}{\pi_0(c)}
 =\frac{c_i(N-1+c_j)}{(c_j+1)(N-2+c_i)}
 =\frac{L_0(c,c')}{L_0(c',c)}.
 \label{eq:app-neutral-detailed-balance}
\end{equation}
This proves stationarity, and irreducibility gives uniqueness.

Define $\Phi_r(f)=\sum_c\pi_r(c)f(c)$.  Then
$\Phi_0(1)=1$ and $\Phi_r(1)=0$ for $r\ge1$.
Multiplying \eqref{eq:app-neutral-law} by factorials and shifting
the count variables gives, for $A=a_0+a_1+a_2\le N$,
\begin{equation}
 \Phi_0\left(\prod_{i=0}^2(c_i)_{\underline{a_i}}\right)
 =(N)_{\underline A}
 \frac{\prod_{i=0}^2\alpha^{\overline{a_i}}}
 {(3\alpha)^{\overline A}}.
 \label{eq:app-neutral-factorial-moments}
\end{equation}
Indeed, set $c'_i=c_i-a_i$ and use
$\alpha^{\overline{c_i}}
=\alpha^{\overline{a_i}}(\alpha+a_i)^{\overline{c'_i}}$.
The sum over $|c'|=N-A$ is evaluated by
\eqref{eq:app-rising-convolution}, with parameters
$\alpha+a_i$.  The resulting
$(3\alpha+A)^{\overline{N-A}}$ cancels the last $N-A$ factors
in the original denominator.

For ordinary powers, define integers $\mathfrak s_{ra}$ by
\begin{equation}
\begin{split}
 &\mathfrak s_{00}=1,\qquad \mathfrak s_{r0}=0\quad(r>0),
 \qquad \mathfrak s_{ra}=0\quad(a<0\text{ or }a>r),\\
 &\mathfrak s_{r+1,a}=a\mathfrak s_{ra}+\mathfrak s_{r,a-1}.
\end{split}
\label{eq:app-power-factorial-recursion}
\end{equation}
The identity
$t(t)_{\underline a}=(t)_{\underline{a+1}}+a(t)_{\underline a}$
proves by induction that
$t^r=\sum_{a=0}^r\mathfrak s_{ra}(t)_{\underline a}$.
With $x=c_0$ and $y=c_2$, it follows that
\begin{equation}
 \Phi_0(x^ry^s)=
 \sum_{a=0}^r\sum_{b=0}^s
 \mathfrak s_{ra}\mathfrak s_{sb}(N)_{\underline{a+b}}
 \frac{(N-1)^{\overline a}(N-1)^{\overline b}}
 {[3(N-1)]^{\overline{a+b}}},\qquad r+s\le4.
 \label{eq:app-neutral-ordinary-moments}
\end{equation}
For orders up to four, the coefficients are
\[
\begin{array}{c|rrrrr}
 r\backslash a&0&1&2&3&4\\
\hline
0&1&0&0&0&0\\
1&0&1&0&0&0\\
2&0&1&1&0&0\\
3&0&1&3&1&0\\
4&0&1&7&6&1.
\end{array}
\]
Thus \eqref{eq:app-neutral-ordinary-moments} determines $\Phi_0$
on the polynomials of degree at most four.

\subsection{The stationary equations}

Let $z=N-x-y$ and write
$L_\varepsilon=L_0+\varepsilon L_1$. Table~\ref{tab:app-population-shifts} gives
all six shifts and their two coefficients.

\begin{table}[htbp]
\centering
\begin{tabular}{cccc}
\hline
event & $(\delta_x,\delta_y)$ & $\ell_0(x,y)$ & $\ell_1(x,y)$\\
\hline
$0\to1$ & $(-1,0)$ & $x+xz/(N-1)$ & $-xz/(N-1)$\\
$0\to2$ & $(-1,1)$ & $x+xy/(N-1)$ & $-xy/(N-1)$\\
$1\to0$ & $(1,0)$ & $z+zx/(N-1)$ & $0$\\
$1\to2$ & $(0,1)$ & $z+zy/(N-1)$ & $0$\\
$2\to0$ & $(1,-1)$ & $y+yx/(N-1)$ & $yx/(N-1)$\\
$2\to1$ & $(0,-1)$ & $y+yz/(N-1)$ & $yz/(N-1)$\\
\hline
\end{tabular}
\caption{The neutral and first order population shifts.}
\label{tab:app-population-shifts}
\end{table}

For $v=0,1$, the associated operator is
\begin{equation}
 L_vf(x,y)=\sum_{\text{six rows}}
 \ell_v(x,y)[f(x+\delta_x,y+\delta_y)-f(x,y)].
 \label{eq:app-shift-generator}
\end{equation}

Let $\mathcal P_d$ denote the polynomials in $x,y$ of total
degree at most $d$.  Finite differences in
\eqref{eq:app-shift-generator} give
\begin{equation}
 L_0\mathcal P_d\subseteq\mathcal P_d,
 \qquad L_1\mathcal P_d\subseteq\mathcal P_{d+1}.
 \label{eq:app-degree-bounds}
\end{equation}
On the highest homogeneous part of degree $d\ge1$, $L_0$ acts
by the scalar
\begin{equation}
 \lambda_d=-3d-\frac{d(d-1)}{N-1}
 =-\frac{d(3N+d-4)}{N-1}.
 \label{eq:app-leading-eigenvalues}
\end{equation}
To establish this, the highest degree contribution of natural
motion is $-3(x\partial_x+y\partial_y)$.
The opposite neutral copying shifts cancel the first derivative
terms.  Their second derivative contribution at highest degree is
\[
 -\frac{x^2\partial_{xx}+2xy\partial_{xy}+y^2\partial_{yy}}{N-1}.
\]
Terms of derivative order at least three have degree below $d$.
Euler's identities for homogeneous polynomials now give
\eqref{eq:app-leading-eigenvalues}.  For $d=1,2,3,4$, the four
scalars are respectively
\[
 -3,\qquad -\frac{2(3N-2)}{N-1},\qquad
 -\frac{3(3N-1)}{N-1},\qquad -\frac{12N}{N-1}.
\]

Evaluation of $\mathcal P_d$ on the population simplex is
injective for $d\le N$.  Order the factorial basis
$(x)_{\underline a}(y)_{\underline b}$ and the evaluation points
$(a,b)$ by total degree.  A basis element vanishes at a point
of smaller total degree.  Within a fixed degree, it vanishes
unless the basis indices equal the point, where its value is
$a!b!$.  The evaluation matrix on $x+y\le d$ is thus triangular
with nonzero diagonal.  In particular, the fifteen monomials
of $\mathcal P_4$ are linearly independent as population
functions for every $N\ge4$.

\begin{lemma}
\label{lem:app-polynomial-poisson}
For $f\in\mathcal P_d$, $d\le4$, there is a unique
$\Psi_f\in\mathcal P_d$ with zero constant coefficient such that
\begin{equation}
 L_0\Psi_f=f-\Phi_0(f).
 \label{eq:app-poisson-equation}
\end{equation}
It is obtained by the following finite descent.  Initialize
$r=f-\Phi_0(f)$ and $\Psi=0$.  For $j=d,d-1,\ldots,1$, let
$[r]_j$ be the degree $j$ homogeneous part of $r$ and set
\begin{equation}
 \psi_j=\frac{[r]_j}{\lambda_j},\qquad
 \Psi\leftarrow\Psi+\psi_j,\qquad
 r\leftarrow r-L_0\psi_j.
 \label{eq:app-poisson-descent}
\end{equation}
\end{lemma}

\begin{proof}
Equation~\eqref{eq:app-leading-eigenvalues} removes the degree
$j$ part at the indicated step.  Equation
\eqref{eq:app-degree-bounds} ensures that no higher degree is
introduced.  At the end the remainder is constant.  It has
mean zero because its initial mean is zero and
$\Phi_0(L_0\psi_j)=0$ at every step.  The constant remainder
is therefore zero.  Each $\psi_j$ has zero constant coefficient.
For uniqueness, the difference of two solutions is annihilated
by $L_0$, and irreducibility makes it constant; its prescribed
constant coefficient is zero.
\end{proof}

For example, $L_0x=N-3x$ and $\Phi_0(x)=N/3$, so
$\Psi_x=-x/3$.  Taking Taylor coefficients in
$\pi_\varepsilon L_\varepsilon=0$ gives
$\pi_rL_0=-\pi_{r-1}L_1$.  Apply this identity to
\eqref{eq:app-poisson-equation} and use $\Phi_r(1)=0$ to obtain
\begin{equation}
 \Phi_r(f)=-\Phi_{r-1}(L_1\Psi_f).
 \label{eq:app-stationary-recursion}
\end{equation}
With
\begin{equation}
 \mathcal Rf=-L_1\Psi_f,
 \qquad \Phi_r(f)=\Phi_0(\mathcal R^rf),
 \label{eq:app-response-operator}
\end{equation}
we can express the stationary response in terms of neutral
moments.  The second identity holds whenever the successive
Poisson equations have inputs of degree at most four.  Since
$L_1$ annihilates constants, the response is independent of the
additive constant in the Poisson solution.

To write the descent in coordinates, let
$f=\sum_{a+b\le d}f_{ab}x^ay^b$.  Formula
\eqref{eq:app-shift-generator} acts on each monomial by
\begin{equation}
 L_v(x^ay^b)=\sum_{\text{six rows}}\ell_v
 \sum_{\substack{0\le p\le a,\ 0\le q\le b\\(p,q)\ne(a,b)}}
 \binom ap\binom bq
 x^py^q\delta_x^{a-p}\delta_y^{b-q}.
 \label{eq:app-monomial-generator}
\end{equation}
Equations~\eqref{eq:app-neutral-ordinary-moments},
\eqref{eq:app-poisson-descent}, and
\eqref{eq:app-monomial-generator} determine its coefficients by
finite sums and division by the four nonzero $\lambda_j$.

\subsection{The cubic coefficient}

For $r=0,1,2,3$ define
\begin{equation}
 G_r(x,y)=x h^{[r]}(2,c-e_0)-y h^{[r]}(0,c-e_2),
 \qquad c=(x,N-x-y,y).
 \label{eq:app-transport-polynomials}
\end{equation}
The two substitutions in the founder formulas are
\begin{equation}
\begin{array}{c|ccc}
 &S&R&s_i\\
\hline
 h^{[r]}(2,c-e_0)&y-x+1&x+y-1&1\\
 h^{[r]}(0,c-e_2)&y-x-1&x+y-1&-1.
\end{array}
\label{eq:app-transplant-substitutions}
\end{equation}
These are polynomials of degree at most $r+1$.  At $x=0$ or
$y=0$, the corresponding prefactor vanishes, so they agree with
the transplant expression on the entire population simplex.

Put $t=y-x$ and $u=x+y$.  Expanding the founder formulas gives
\begin{equation}
\begin{aligned}
 G_0={}&-\frac{t}{N},\\
 G_1={}&-\frac{t^2+(N-2)u}{N(3N-1)},\\
 G_2={}&t\bigl[-a_2t^2+(2a_2+c_2-b_2)u\\
 &\qquad\quad-a_2+b_2-c_2-\delta_2\bigr],\\
 G_3={}&-\beta_0t^4
 +(3\beta_0+\beta_3-\beta_1)ut^2\\
 &\quad+(-3\beta_0+\beta_1-\beta_2-2\beta_3-\beta_6)t^2\\
 &\quad +(\beta_1+\beta_4)u^2
 +(\beta_0-\beta_1+\beta_2+\beta_3-\beta_4
       +\beta_5+\beta_6)u.
\end{aligned}
\label{eq:app-expanded-transport-polynomials}
\end{equation}
Expansion of the stationary transport coefficient gives
\begin{equation}
 [\varepsilon^3]\Delta_N
 =\Phi_0(G_3)+\Phi_1(G_2)+\Phi_2(G_1)+\Phi_3(G_0).
 \label{eq:app-cubic-decomposition}
\end{equation}
The four terms are exactly
\begin{equation}
\begin{aligned}
 \Phi_0(G_3)
 &=\frac{2(N-1)P_0}{9N(3N-2)^2(3N-1)^3(3N+1)},\\
 \Phi_1(G_2)
 &=\frac{2(N-1)^2P_1}{9N(3N-2)^2(3N-1)^3},\\
 \Phi_2(G_1)
 &=-\frac{2(N-1)^2P_2}{9(3N-2)^3(3N-1)^3},\\
 \Phi_3(G_0)
 &=-\frac{2(N-1)P_3}{9(3N-2)^3(3N-1)^2},
\end{aligned}
\label{eq:app-four-contributions}
\end{equation}
where
\begin{equation}
\begin{aligned}
 P_0={}&27N^6-144N^5+315N^4-335N^3+42N^2+79N-16,\\
 P_1={}&27N^4-66N^3+17N^2+34N-10,\\
 P_2={}&81N^4-279N^3+369N^2-236N+57,\\
 P_3={}&9N^4-45N^3+71N^2-54N+15.
\end{aligned}
\label{eq:app-contribution-numerators}
\end{equation}

We now compute the four terms.  The Poisson solutions we need are
\begin{equation}
\begin{aligned}
 \Psi_t={}&-\frac t3,&\Psi_u={}&-\frac u3,\\
 \Psi_{t^2}={}&-\frac{N-1}{2(3N-2)}t^2
              -\frac{5N-3}{6(3N-2)}u,\\
 \Psi_{tu}={}&-\frac{N-1}{2(3N-2)}tu
              -\frac{(N+1)(2N-1)}{6(3N-2)}t,\\
 \Psi_{t^3}={}&-\frac{N-1}{3(3N-1)}t^3
              -\frac{(N-1)(5N-3)}{2(3N-2)(3N-1)}tu\\
 &\quad-\frac{22N^3-39N^2+30N-9}{6(3N-2)(3N-1)}t.
\end{aligned}
\label{eq:app-five-poisson-solutions}
\end{equation}
Applying the six shifts to the right-hand sides gives
$f-\Phi_0(f)$ in each case.  The required neutral moments are
\begin{equation}
\begin{aligned}
 \Phi_0(u)&=\frac{2N}{3},&
 \Phi_0(t^2)&=\frac{2N(4N-3)}{3(3N-2)},\\
 \Phi_0(u^2)&=\frac{2N(2N^2-1)}{3(3N-2)},\\
 \Phi_0(ut^2)&=
 \frac{2N(N+1)(2N-1)(4N-3)}{3(3N-2)(3N-1)},\\
 \Phi_0(t^4)&=
 \frac{2N(4N-3)(8N^2-9N+3)}{3(3N-2)(3N-1)}.
\end{aligned}
\label{eq:app-five-neutral-moments}
\end{equation}
They follow from \eqref{eq:app-neutral-ordinary-moments} after
expanding $t=y-x$ and $u=x+y$.

For the response operator, set
\begin{equation}
\begin{aligned}
 \xi_0&=\frac1{3(N-1)},&
 \xi_1&=-\frac N{3(N-1)},\\
 \zeta_0&=\frac1{3N-2},&
 \zeta_1&=-\frac{N(3N-4)}{3(N-1)(3N-2)},\\
 \zeta_2&=-\frac{N^2}{3(N-1)(3N-2)}.
\end{aligned}
\label{eq:app-reduced-response-coefficients}
\end{equation}
The first three Poisson solutions and the shift table give
\begin{equation}
\begin{aligned}
 \mathcal Rt&=\xi_0t^2+\xi_1u,\\
 \mathcal Ru&=\xi_0tu+\xi_1t,\\
 \mathcal R(t^2)&=\zeta_0t^3+\zeta_1tu+\zeta_2t.
\end{aligned}
 \label{eq:app-reduced-response-map}
\end{equation}
The remaining responses enter only through their neutral means.
Let $K_1=\Phi_0(\mathcal Rt)$,
$K_2=\Phi_0(\mathcal R(tu))$, and
$K_3=\Phi_0(\mathcal R(t^3))$.  Applying $-L_1$ to
\eqref{eq:app-five-poisson-solutions} and using
\eqref{eq:app-five-neutral-moments} gives
\begin{equation}
\begin{aligned}
 K_1&=-\frac{2N(N-1)}{3(3N-2)},\\
 K_2&=-\frac{2N(N-1)(6N^3-N^2-6N+3)}
 {3(3N-2)^2(3N-1)},\\
 K_3&=-\frac{2N(N-1)(72N^4-165N^3+157N^2-75N+15)}
 {3(3N-2)^2(3N-1)^2}.
\end{aligned}
\label{eq:app-response-means}
\end{equation}
For the first formula, we obtain
$K_1=\xi_0\Phi_0(t^2)+\xi_1\Phi_0(u)$.
For $K_2$ and $K_3$, we apply the same shifts to $\Psi_{tu}$
and $\Psi_{t^3}$, respectively.

The first row of \eqref{eq:app-four-contributions} follows
by taking the five moments in
\eqref{eq:app-five-neutral-moments} in the last line of
\eqref{eq:app-expanded-transport-polynomials}.
For the remaining rows,
\eqref{eq:app-response-operator} and
\eqref{eq:app-reduced-response-map} give the explicit scalar
identities
\begin{equation}
\begin{aligned}
 \Phi_1(G_2)={}&-a_2K_3+(2a_2+c_2-b_2)K_2\\
 &\quad+(-a_2+b_2-c_2-\delta_2)K_1,\\
 \Phi_2(G_1)={}&-\frac{1}{N(3N-1)}\bigl[
 \zeta_0K_3+(\zeta_1+(N-2)\xi_0)K_2\\
 &\hspace{31mm}+(\zeta_2+(N-2)\xi_1)K_1\bigr],\\
 \Phi_3(G_0)={}&-\frac1N\bigl[
 \xi_0\zeta_0K_3+\xi_0(\zeta_1+\xi_1)K_2\\
 &\hspace{14mm}+(\xi_0\zeta_2+\xi_1^2)K_1\bigr].
\end{aligned}
\label{eq:app-scalar-contribution-calculation}
\end{equation}
Substitution of \eqref{eq:app-second-founder-coefficients},
\eqref{eq:app-reduced-response-coefficients}, and
\eqref{eq:app-response-means}, followed by a common denominator
in each row, yields the last three fractions in
\eqref{eq:app-four-contributions}.

In the coordinates $x,y$, the same calculation takes the form
\begin{equation}
\begin{split}
 \Phi_{j}(G_{3-j})
 ={}&\sum_{a+b\le4}[x^ay^b](\mathcal R^jG_{3-j})
 \sum_{p=0}^a\sum_{q=0}^b
 \mathfrak s_{ap}\mathfrak s_{bq}\\
 &\qquad\times(N)_{\underline{p+q}}
 \frac{(N-1)^{\overline p}(N-1)^{\overline q}}
 {[3(N-1)]^{\overline{p+q}}},\qquad 0\le j\le3.
\end{split}
 \label{eq:app-contribution-finite-sums}
\end{equation}
The degree paths are
\[
 G_3:\ 4;\qquad
 G_2:\ 3\to4;\qquad
 G_1:\ 2\to3\to4;\qquad
 G_0:\ 1\to2\to3\to4.
\]
Every input to $L_1$ has degree at most three, and the final
neutral moment has degree at most four.  Thus all Poisson
equations used in the calculation fall within
Lemma~\ref{lem:app-polynomial-poisson}.

The common denominator in
\eqref{eq:app-cubic-decomposition} is
$9N(3N-2)^3(3N-1)^3(3N+1)$.  Its numerator simplifies by
the integer polynomial identity
\begin{equation}
\begin{split}
 &2(N-1)(3N-2)P_0
 +2(N-1)^2(3N-2)(3N+1)P_1\\
 &\quad-2N(N-1)^2(3N+1)P_2
 -2N(N-1)(3N-1)(3N+1)P_3\\
 &\hspace{12mm}=-2(N-3)(N-2)(N-1)^2(3N-2)(3N-1).
\end{split}
\label{eq:app-cubic-factorization}
\end{equation}
Combining the four terms and cancelling the two displayed
factors proves
\begin{equation}
 [\varepsilon^3]\Delta_N
 =-\frac{2(N-3)(N-2)(N-1)^2}
 {9N(3N-2)^2(3N-1)^2(3N+1)}.
 \label{eq:app-cubic-coefficient}
\end{equation}

Finally, symmetry of the neutral law gives $\Phi_0(G_0)=0$.
The moments in \eqref{eq:app-five-neutral-moments} give
$\Phi_0(G_1)=-2(N-1)/[3(3N-2)]$.
Also $\Psi_{G_0}=t/(3N)$, and the shift table gives
\begin{equation}
 \mathcal RG_0=\frac{Nu-t^2}{3N(N-1)},\qquad
 \Phi_1(G_0)=\frac{2(N-1)}{3(3N-2)}.
 \label{eq:app-linear-cancellation}
\end{equation}
The two first order terms cancel.  Exchanging sites $0$ and
$2$ replaces $\varepsilon$ by $-\varepsilon$ and reverses
the port difference.  Uniqueness of the stationary and
founder solutions gives
$\Delta_N(-\varepsilon)=-\Delta_N(\varepsilon)$.
All even Taylor coefficients vanish.  The resulting expansion is
\begin{equation}
 \Delta_N(\varepsilon)=-C_N\varepsilon^3+O(\varepsilon^5),
 \qquad
 C_N=\frac{2(N-3)(N-2)(N-1)^2}
 {9N(3N-2)^2(3N-1)^2(3N+1)}>0.
 \label{eq:app-cubic-expansion}
\end{equation}
For each fixed $N\ge4$, this gives a strictly negative
transport for all sufficiently small positive $\varepsilon$.

The Python script in Supplementary materials checks the
founder identities \eqref{eq:app-founder-residuals} and
\eqref{eq:app-founder-anchors}, reconstructs the four contributions
\eqref{eq:app-four-contributions} from the shifts
\eqref{eq:app-shift-generator}, the moments
\eqref{eq:app-neutral-ordinary-moments}, and the Poisson descent
\eqref{eq:app-poisson-descent}, and verifies the factorization
\eqref{eq:app-cubic-factorization}. All checks use exact arithmetic
over $\mathbb Q(N)$, with $N$ an indeterminate.

\clearpage
\section{Finite mixing and sparse proposal estimates}\label{app:estimates}

We use the supremum norm for functions and the induced norm for
operators. A probability $\pi$ acts
on a function by $\pi f=\sum_c\pi(c)f(c)$.

\subsection{A centred inverse for the population chain}

\begin{lemma}\label{lem:finite-mixing}
For a finite irreducible generator $L$ with stationary probability $\pi$,
the projection $\Pi f=(\pi f)\1$ satisfies
\begin{equation}\label{eq:finite-mixing}
\begin{gathered}
 \|e^{tL}-\Pi\|\le C e^{-\eta t},\qquad
 B_0=\int_0^\infty(e^{tL}-\Pi)\,dt,\\
 LB_0=B_0L=\Pi-I
\end{gathered}
\end{equation}
for some $C,\eta>0$.
\end{lemma}

\begin{proof}
The assertion is immediate when there is one state, with $B_0=0$.
Suppose there are $M\ge2$ states. Choose
$\Lambda>\max_c[-L(c,c)]$ and put
$K=I+L/\Lambda$. This stochastic matrix has a positive diagonal.
Any two states are connected by a positive path of length at most
$r=M-1$. Since the diagonal is positive, we can extend each path
to length $r$, and therefore $K^r(c,d)>0$ for all $c,d$.
Let $\epsilon=\min_{c,d}K^r(c,d)>0$ and
$\rho=\min(M\epsilon,1/2)$. Each row of $K^r$ is a mixture of the
uniform probability, with weight $\rho$, and another probability.
It follows that right multiplication by $K^r$ contracts differences
of probability measures in total variation by the factor $1-\rho$.

Ces\`aro averages of any initial distribution have a stationary
subsequential limit in the finite probability simplex. The contraction
gives uniqueness and exponential convergence to that limit. Its
coordinates are strictly positive because $K^r$ is strictly positive.
Uniformization,
\[
 e^{tL}=e^{-\Lambda t}\sum_{n\ge0}\frac{(\Lambda t)^n}{n!}K^n,
\]
gives the exponential bound in \eqref{eq:finite-mixing}, and the
integral defining $B_0$ converges. Integrating the derivative of
$e^{tL}$ gives $LB_0=B_0L=\Pi-I$.
\end{proof}

\subsection{Sparse proposals}

\begin{lemma}\label{lem:sparse-resolvent}
Fix a finite cell, $N$ and two distinct ports. For a compact interval
$J\subset(0,1)$, let
\[
 v_\theta(c)=\theta c_\alpha+(1-\theta)c_\beta,
 \qquad R_{\lambda,\theta}=(\lambda M_{v_\theta}-L_N)^{-1}.
\]
The inverse exists for every $\lambda>0$ and $\theta\in J$. There are
$\lambda_0,C_R>0$ such that
\begin{equation}\label{eq:sparse-resolvent-bound}
 \left\|\lambda R_{\lambda,\theta}g
 -\frac{\pi_Ng}{\pi_Nv_\theta}\1\right\|\le C_R\lambda
\end{equation}
for $0<\lambda\le\lambda_0$, $\theta\in J$ and $\|g\|\le N$.
In particular, $\lambda R_{\lambda,\theta}\1$ is uniformly bounded.
\end{lemma}

\begin{proof}
We first show that the inverse exists. The function $v_\theta$ is
nonnegative and is positive at some count state. From every state,
irreducibility gives a path to such a state, followed by killing
at rate $\lambda v_\theta$, with positive probability. Since there
are finitely many states, killing occurs with probability bounded
away from zero in a fixed time interval. By the Markov property,
the survival semigroup decays exponentially. Its time integral is
$R_{\lambda,\theta}$.

Apply Lemma~\ref{lem:finite-mixing} with $L=L_N$. Put
$K_0\ge\max(1,\|B_0\|)$ and
\[
 v_* = \inf_{\theta\in J}\pi_Nv_\theta>0.
\]
The stationary mean count is positive at each port, so $v_*>0$.
We also have $\|v_\theta\|\le N$.
Decompose
\[
 u=\lambda R_{\lambda,\theta}g=\xi\1+w,
 \qquad \xi=\pi_Nu,\quad \pi_Nw=0.
\]
The resolvent equation gives the first identity below. Applying
$B_0$ and taking the stationary mean give the remaining two:
\begin{equation}\label{eq:resolvent-decomposition}
\begin{aligned}
 L_Nu&=\lambda(v_\theta u-g),\\
 w&=\lambda B_0\bigl(g-v_\theta(\xi\1+w)\bigr),\\
 \xi\pi_Nv_\theta+\pi_N(v_\theta w)&=\pi_Ng.
\end{aligned}
\end{equation}
The last equation yields
$|\xi|\le(N/v_*)(1+\|w\|)$. Hence, with
$A_N=K_0N(1+N/v_*)$,
\[
 \|w\|\le\lambda A_N(1+\|w\|).
\]
For $\lambda_0=(2A_N)^{-1}$ this implies $\|w\|\le2\lambda A_N$.
The same stationary equation gives
\[
 \left|\xi-\frac{\pi_Ng}{\pi_Nv_\theta}\right|
 \le\frac{N}{v_*}\|w\|.
\]
Take $C_R=2A_N(1+N/v_*)$ to obtain
\eqref{eq:sparse-resolvent-bound}.
\end{proof}

Since the first proposal occurs almost surely, its interpretation
in Section~\ref{sec:transplants} gives
$\lambda R_{\lambda,\theta}v_\theta=\1$. The lemma also implies
\[
 \E_c T_\lambda=R_{\lambda,\theta}\1(c)
 =\frac{1}{\lambda\pi_Nv_\theta}+O(1)
\]
uniformly in $c$ and $\theta\in J$, where $T_\lambda$ is the first
proposal time. Thus the leading terms for its direction and waiting
time depend only on the stationary population law.

\subsection{Colour speeds}

Let mutant particles have speed $t>0$ and residents speed one.
The coloured generator is obtained from
\eqref{eq:two-colour-rates} by multiplying each rate by $t$ when
the moving or dying category $r$ is mutant. Resident rates are
unchanged. This is the process in
\eqref{eq:speed-port-coefficients}; write $B_N(t)$ for its
transient matrix and $r_N(t)$ for its rates to all-mutant fixation.

On a compact interval $K\subset(0,\infty)$ the marked total rate has a
uniform upper bound, while every specified copying event has rate at
least
\[
 \frac{\min(1,\inf K)\min_i k_i}{N-1}>0.
\]
Applying the proof of Lemma~\ref{lem:fixation} with these bounds
gives a constant $p_K>0$ such that, for every integer $r\ge0$,
\begin{equation}\label{eq:speed-fixation-tail}
\begin{gathered}
 \sup_{t\in K}\sup_z\Prob_z(\Theta_t>r)\le(1-p_K)^r,
 \\ \sup_{t\in K}\sup_z\E_z\Theta_t\le1/p_K,
\end{gathered}
\end{equation}
where the suprema over $z$ run through mixed colour configurations.
The entries of $B_N(t),r_N(t)$ are affine in $t$ and $B_N(t)$ is
invertible for every positive $t$. Therefore
\[
 f_t(j,d)=\bigl[(-B_N(t))^{-1}r_N(t)\bigr](e_j,d)
\]
is rational and has no singularities on $(0,\infty)$. This proves the
fixation continuity used in Remark~\ref{rem:colour-speeds};
\eqref{eq:speed-fixation-tail} gives uniform duration estimates on
compact speed intervals.

\end{document}